\documentclass[11pt, a4paper]{article}
\usepackage{amsmath, amsfonts}
\usepackage[usenames]{color}
\usepackage[english]{babel}

\usepackage[table, dvipsnames]{xcolor}
\usepackage{array}
\newcolumntype{P}[1]{>{\centering\arraybackslash}p{#1}}
\newcolumntype{M}[1]{>{\centering\arraybackslash}m{#1}}
\usepackage{multirow}
\usepackage{tabularx}
\usepackage[all]{xy}
\usepackage{tikz-cd}
\usepackage{tikz}
\usepackage{float}
\usepackage{xcolor}
\usepackage{graphicx}
\usepackage{geometry}
\usepackage{mathtools}
\usepackage{mathrsfs}
\usepackage{tcolorbox}
\usepackage[table]{xcolor}
\usepackage{hyperref} 
\hypersetup{bookmarks = true, 
	colorlinks = true, 
	linkcolor = blue, 
	citecolor = blue, 
}
\usepackage{tcolorbox}
\definecolor{mygray}{rgb}{0.98,0.95,0.75}
\newtcolorbox{mybox}{
	arc=0pt,
	boxrule=0pt,
	colback=mygray,
	width=1\textwidth,   
	colupper=black,
	fontupper=\normalsize
}

\newcommand{\xdownarrow}[1]{%
	{\left\downarrow\vbox to #1{}\right.\kern-\nulldelimiterspace}
}

\newcommand{\xuparrow}[1]{%
	{\left\uparrow\vbox to #1{}\right.\kern-\nulldelimiterspace}
}

\usepackage{graphicx}
\usepackage{amssymb}
\usepackage{amsmath}

\newtheorem{theorem}{Theorem}[section]

\newtheorem{conjecture}[theorem]{Conjecture}
\newtheorem{corollary}[theorem]{Corollary}

\newtheorem{definition}[theorem]{Definition}
\newtheorem{example}[theorem]{Example}

\newtheorem{lemma}[theorem]{Lemma}

\newtheorem{proposition}[theorem]{Proposition}
\newtheorem{remark}[theorem]{Remark}

\newenvironment{proof}[1][Proof]{\noindent\textbf{#1.} }{\ \rule{0.5em}{0.5em}}

\title{\textbf{Zariski's multiplicity conjecture for quasihomogeneous hypersurfaces with non-isolated singularities}}
\author{ \ \ \ \\{Silva, O. N. $ \ \ $ and  $ \ \ $ Silva Jr, M. M. $ \ \ $}}
\date{}

\begin{document}
	
	\maketitle
	
	\begin{abstract}
		
		In this work, we consider a pair $(\textbf{X},0)$ and $(\textbf{Y},0)$ of hypersurfaces in $(\mathbb{C}^{n+1},0)$ parametrized by finitely determined, quasihomogeneous map germs $f$ and $g,$ respectively. Zariski asked whether the multiplicity is preserved under topological equivalence of hypersurface germs. We address this question within a wide class of $n$-dimensional quasihomogeneous varieties with non-isolated singularities in $\mathbb{C}^{n+1},$ where $2\le n\le 4.$ This class consists of varieties that arise as image of finitely determined, quasihomogeneous map germs. Using a quasihomogeneous normal form, we derive explicit formulas for the multiplicity in terms of the weights and the degrees of the map germ. Our results show that multiplicity, within this setting, is determined by the weighted data and is invariant under topological equivalence, thereby confirming Zariski's multiplicity conjecture and extending current knowledge beyond the isolated singularity case.
	\end{abstract}

	
	
	\section{Introduction}
	
	$ \ \ \ \ $ In his discourse upon leaving the presidency of the American Mathematical Society in 1971, Zariski proposed some questions in singularity theory (see \cite{Conjzari}). One of them is what is nowadays known as \textit{Zariski's multiplicity conjecture} whose statement we will explain in the sequel: consider a germ $F:(\mathbb{C}^n,0)\rightarrow (\mathbb{C},0)$ of a reduced analytic function at the origin with $F \not\equiv 0$. Let $(V(F),0)$ be the germ of the zero set of $F$ at the origin. Write

\begin{center}
$F = F_m + F_{m+1} + \cdots + F_k + \cdots$
\end{center}

\noindent where each $F_k$ is a homogeneous polynomial of degree $k$ and $F_m \not \equiv 0$. The integer $m$ is called the multiplicity of $V(F)$ at $0$ and it is denoted by $m(V(F),0)$ (or $m(V(F))$ for short).

We will establish a notation that will be used throughout the work. Let $F,G:(\mathbb{C}^n,0)\rightarrow (\mathbb{C},0)$ be reduced germs (at the origin) of holomorphic functions, $(\textbf{X},0):=(V(F),0)$, $(\textbf{Y},0):=(V(G),0)$ the corresponding germs of hypersurfaces in $(\mathbb{C}^n,0)$. We say that $F$ and $G$ are topologically equivalent if there is a germ of homeomorphism $\varphi:(\mathbb{C}^n,0)\rightarrow (\mathbb{C}^n,0)$ such that $\varphi(\textbf{X},0)=(\textbf{Y},0)$. 

\begin{mybox}
\textbf{Zariski's multiplicity conjecture:} Consider $F,G, \textbf{X}$ and $\textbf{Y}$ as above. If $F$ and $G$ are topologically equivalent, then the multiplicities $m(\textbf{X},0)$ and $m(\textbf{Y},0)$ are equal.
\end{mybox}

More than $50$ years later, Zariski's multiplicity conjecture above is still open. Several special cases of the conjecture have been proved.. We would like to cite some historical cases that will be important throughout this work. The first contribution to prove Zariski's multiplicity conjecture was given by Zariski in \cite{Zariskii}, where he proved the conjecture when $n=2$. In 1973, due to the works of A'Campo and Lê published in this same year, one can deduce a proof for the 
case in which $m(\textbf{X},0)=1$ combining the results \cite[Th. 3]{ACampo} with \cite[Prop.]{Le}. As an application, it is not hard to see that if $F$ and $G$ are topologically equivalent and $(\textbf{X},0)$ has an isolated singularity at the origin, then so does $(\textbf{Y},0)$. In 1980, Navarro Aznar showed the conjecture (see \cite{Navarro}) in the case where $n=3$ and $m(\textbf{X},0)=2$.

Among recent advances, it is worth highlighting the important result obtained by Fernández de Bobadilla and Pełka \cite{Pelka}, which proves that if a family $\textbf{X}_t$ of hypersurfaces with isolated singularities is $\mu$-constant, then $\textbf{X}_t$ is equimultiple. For an overview of the conjecture, the reader may also consult the work of Eyral \cite{Eyral}.

 There are proofs of other cases of Zariski's multiplicity conjecture under additional conditions. One could, for instance, add the hypothesis that $F$ and $G$ are both quasihomogeneous with an isolated critical point at the origin. In this direction, in 1988 Saeki (\cite{Saeki}) and independently Xu and Yau (\cite{Xu} and \cite{Yau}) presented a proof for the conjecture in the case where $n=3$ and $F$ and $G$ are both quasihomogeneous with an isolated critical point at the origin.
 
 Note that if $F$ has an isolated critical point at the origin then $(\textbf{X},0)$ has isolated singularity. If $F$ and $G$ are quasihomogeneous surfaces in $\mathbb{C}^3$ that are topologically equivalent and $(\textbf{X},0)$ has a non-isolated singularity at the origin, then we cannot apply the Saeki-Xu-Yau result to conclude that $(\textbf{X},0)$ and $(\textbf{Y},0)$ have the same multiplicity at the origin. One of the main goals of this work is precisely to address the case of non-isolated singularities not covered by the Saeki-Xu-Yau result. We present sufficient conditions in this setting to conclude that the multiplicities $m(\textbf{X},0)$ and $m(\textbf{Y},0)$ are the same.  More precisely, our first main result is the following one:
 
\begin{theorem}\label{mainresult} Let $(\textbf{X},0)$ and $(\textbf{Y},0)$ be germs of irreducible surfaces in $(\mathbb{C}^3,0)$ defined as the zero set of reduced quasihomogeneous map germs $F,G:(\mathbb{C}^3,0)\rightarrow (\mathbb{C},0)$, respectively. Suppose that $(\textbf{X},0)$ (respectively $(\textbf{Y},0)$) has a smooth normalization and outside the origin the only singularities of $(\textbf{X},0)$ (respectively $(\textbf{Y},0)$) are transverse double points. If $(\textbf{X},0)$ and $(\textbf{Y},0)$ are topologically equivalent, then $m(\textbf{X},0)=m(\textbf{Y},0)$.
\end{theorem}  
  
One of the simplest examples of surfaces in $(\mathbb{C}^3,0)$ in this setting is the singularity $C_5$ of Mond's list \rm(\cite[p. 378]{[8]}). It is a surface parametrized by the map $f(x,y)=(x,y^2,xy^3-x^5y)$ (see Figure \ref{figura11}). This surface is given as the zero set of $F(X,Y,Z)=Z^2-X^2Y^3+2X^6Y^2-X^{10}Y$, which is quasihomogeneous of multiplicity $2$. 

\begin{figure}[h]
\centering
\includegraphics[scale=0.21]{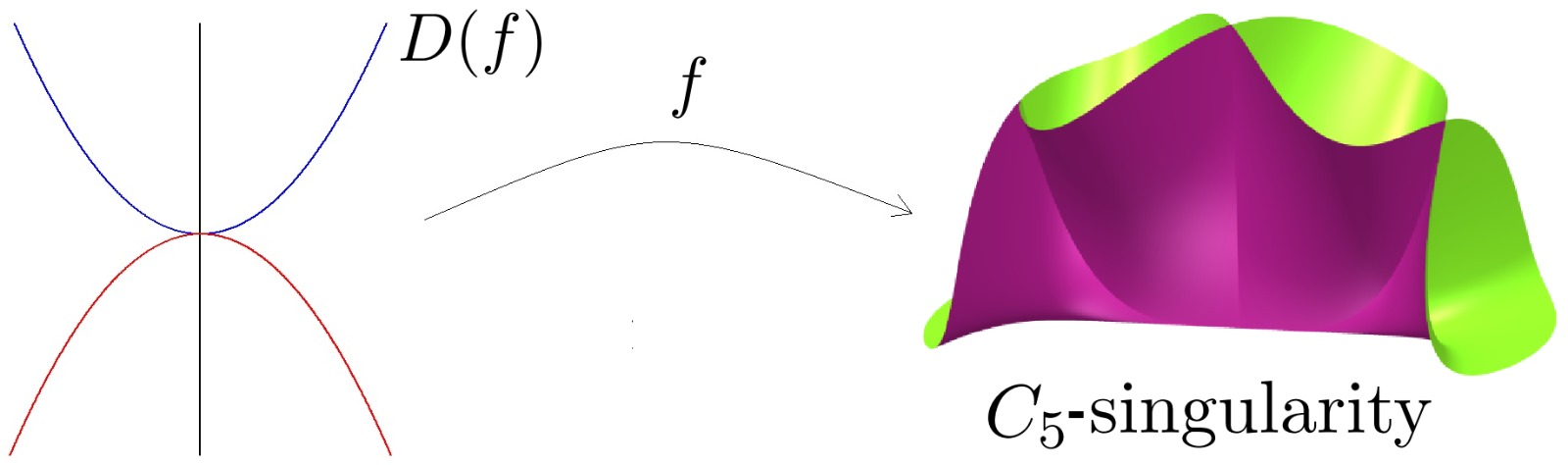} 
\caption{The singularity $C_5$ of Mond and
its double point curve $D(f)$ (real points).}\label{figura11}
\end{figure}

In fact, the image of any quasihomogeneous map germ in Mond's list is an example of a surface satisfying the conditions of Theorem \ref{mainresult}. Outside of Mond's list, we can also find examples of higher multiplicity in the literature satisfying the conditions of Theorem \ref{mainresult}. For example, we have the image of the map germs $f,g:(\mathbb{C}^2,0)\rightarrow (\mathbb{C}^3,0)$ defined by $f(x,y)=(x,y^4,x^5y-5x^3y^3+4xy^5+y^6)$ (\cite[Ex. 5.5]{Ruas}) with multiplicity $4$ and $g(x,y)=(x^n,y^m,(x+y)^k)$ (\cite[Ex. 16]{Guille}) with multiplicity $nm$ (where $n,m$ and $k$ are coprime in pairs and $n<m<k$) and also some classes of examples that appear in \cite{[3]}.
  
In a few words, the technique used to prove Theorem \ref{mainresult} is that under the assumptions we add, both germs of surfaces $(\textbf{X},0)$ and $(\textbf{Y},0)$ admit a special type of parametrization, which we will denote by $f,g$, respectively. The parametrizations $f$ and $g$ can be viewed as quasihomogeneous finitely determined map germs from $(\mathbb{C}^2,0)$ to $(\mathbb{C}^3,0)$ (whose images are $(\textbf{X},0)$ and $(\textbf{Y},0)$, respectively). In this way, we used some tools of the theory of finitely determined map germs to construct a proof for Theorem \ref{mainresult}.

The first part of this paper is devoted to studying the multiplicity of the image of a finite quasihomogeneous map germ from $(\mathbb{C}^2,0)$ to $(\mathbb{C}^3,0)$ (see Section \ref{multiplicidade}). Let $f:(\mathbb{C}^2,0)\rightarrow (\mathbb{C}^3,0)$ be finite map germ of degree one onto its image $(\textbf{X},0):=(f(\mathbb{C}^2),0)$. Consider a map germ $F:(\mathbb{C}^3,0)\rightarrow (\mathbb{C},0)$ such that $(\textbf{X},0)=(V(F),0)$. If $f$ is quasihomogeneous, then we obtain from \cite[Prop. 1.15]{form} that $F$ is quasihomogeneous, with weighted degree given by the product of the degrees of $f$ divided by the product of the weights of $f$. This leads to a natural question:

\begin{mybox}
\textbf{Question 1:} Let $f$ be as above, i.e., a quasihomogeneous finite map germ from $(\mathbb{C}^2,0)$ to $(\mathbb{C}^3,0)$ of degree one onto its image $(\textbf{X},0)$. Is the multiplicity of $(\textbf{X},0)$ determined by the weights and degrees of $f$?
\end{mybox}

Unfortunately, the answer to Question 1 is negative (see Example \ref{exquasihomog}). Note that we do not ask in Question 1 that the only singularities of $(\textbf{X},0)$ outside the origin are transverse double points. In this direction, we show that Question 1 has a positive answer if we assume that $f$ is finitely determined (see Lemma \ref{lema aux5}); this constitutes our second main result. As an application, in Section \ref{secZariski} we apply Lemma \ref{lema aux5} as a fundamental key to prove our main result (Theorem \ref{mainresult}).

 Based on the results in first part of the work, a natural question arises: can we extend the result to the case $n\ge3$? In other words, if we consider finitely determined, quasihomogeneous, singular map germs from $(\mathbb{C}^n,0)$ to $(\mathbb{C}^{n+1},0),$ do the results of Sections \ref{multiplicidade} and \ref{secZariski} hold for $n\ge3?$ The second part of this paper is devoted to extending the results in Sections \ref{multiplicidade} and \ref{secZariski} under some conditions. To do this, the guiding question for carrying out this second stage of the work is the following.

\begin{mybox}
\textbf{Question 2:} Let $f:(\mathbb{C}^n,0)\rightarrow(\mathbb{C}^{n+1},0)$ be a finitely determined, quasihomogeneous, singular map germ, with $n\ge3$. Denote by $(\textbf{X},0)=(f(\mathbb{C}^n),0).$ Is the multiplicity of $(\textbf{X},0)$ determined by the weights and degrees of $f$?
\end{mybox}

 In Section \ref{sec4} we provide a positive answer for Question 2 in corank 1 case (see Theorem \ref{lemamult}). Furthermore, in Section \ref{sec5} when $3\le n\le 4$ we present sufficient conditions on this setting to proof that Zariski's multiplicity conjecture for a pair of $n$-varieties in $(n+1)$-space holds. This is our third main result (see Theorem \ref{T2}).

\begin{theorem}\label{T2}
    Let $f,g:(\mathbb{C}^n,0)\rightarrow(\mathbb{C}^{n+1},0)$ be finitely determined, quasihomogeneous, corank 1 map germs, with $3\le n\le 4$. Denote by $(\textbf{X},0)=(f(\mathbb{C}^n),0)$ and $(\textbf{Y},0)=(g(\mathbb{C}^n),0).$ Suppose that $(\textbf{X},0)$ and $(\textbf{Y},0)$ are topologically equivalent. If the multiple point spaces $D^{n-1}(f)$ and $D^{n-1}(g)$ are smooth, then $m(\textbf{X},0)=m(\textbf{Y},0).$
\end{theorem}


Consider the map germs $f_1:(\mathbb{C}^3,0)\rightarrow(\mathbb{C}^4,0)$ and $f_2:(\mathbb{C}^4,0)\rightarrow(\mathbb{C}^5,0)$ given by $$f_1(x,y,z)=(x,y,z^3+xz,z^4+yz)$$ and $$f_2(x,y,z,w)=(x,y,z,w^2+xw,w^3+yw+zw^2).$$ One can verify that $f_1$ and $f_2$ satisfy the hypotheses of Theorem \ref{T2}. Thus, there exist examples satisfying the assumptions in Theorem \ref{T2}. In particular, the map germ $f_1$ is the $P_1$ singularity in the Houston and Kirk classification (see \cite[Table 1]{Houston}).

\section{Preliminaries}

$ \ \ \ \ $ Throughout the paper, we assume that $f:(\mathbb{C}^n,0)\rightarrow(\mathbb{C}^{n+1},0)$ is a finite, generically one-to-one, holomorphic map germ, unless otherwise stated. Throughout this paper, given a finite map $f:\mathbb{C}^n\rightarrow \mathbb{C}^{n+1}$, $(x_1,\dots,x_n)$ and $(X_1,\dots,X_{n+1})$ are used to denote systems of coordinates in $\mathbb{C}^n$ (source) and $\mathbb{C}^{n+1}$ (target), respectively. We use the standard notation of singularity theory as the reader can find in \cite{[7]}.
	
    \subsection{Multiple point space for map germs}\label{2.1}
	
	$ \ \ \ \ $ Multiple point spaces of a map germ from $(\mathbb{C}^n,0)$ to $(\mathbb{C}^p,0)$ with $n\leq p$ play an important role in the study of its geometry. In this section, we will deal only with double point space, which is a fundamental notion in the setting of this work.

Let us begin with the notion of double point space, denoted by $D^2(f),$ and its projection on the source, denoted by $D(f)$. Roughly speaking, $D^2(f)$ is the set of points $(\textbf{x},\textbf{x}^{'}) \in \mathbb{C}^n \times \mathbb{C}^n$ such that either $\textbf{x} \neq \textbf{x}^{'}$ and $f(\textbf{x})=f(\textbf{x}^{'}),$ or $\textbf{x}$ is a singular point of $f$. In order to see $D^2(f)$ as an analytic space, we need to present an appropriate analytic structure for it. We will follow the construction of \cite{[8]}, which is valid for holomorphic maps from $\mathbb{C}^n$ to $\mathbb{C}^p$, with $n\leq p$. 

Let us denote the diagonals of $\mathbb{C}^n \times \mathbb{C}^n$ and $\mathbb{C}^{p} \times \mathbb{C}^{p}$ by $\Delta_{\mathbb{C}^n}$ and $\Delta_{\mathbb{C}^{p}}$, respectively, and denote the sheaves of ideals defining them by $\mathcal{I}_n$ and $\mathcal{I}_{p}$, respectively. We write the points of $\mathbb{C}^n \times \mathbb{C}^n$ and $\mathbb{C}^{p} \times \mathbb{C}^{p}$ as $(\underline{x},\underline{x}^{'})$ and $(\underline{X},\underline{X}^{'})$, respectively.
Locally, 

\begin{center}
$\mathcal{I}_{n}=\langle x_1-x_1^{'},\cdots, x_n-x_n^{'} \rangle$ and $\mathcal{I}_{p}=\langle X_1-X_1^{'},\cdots, X_{p}-X_{p}^{'} \rangle$. 
\end{center}

Since the pull-back $(f \times f)^{\ast}\mathcal{I}_{p}$ is contained in $\mathcal{I}_{n}$ , there exist $\alpha_{ij}\in \mathcal{O}_{\mathbb{C}^{n} \times \mathbb{C}^n}$, such that
\[
f_{i}(\underline{x})-f_{i}(\underline{x}^{'})= \alpha_{i1}(\underline{x},\underline{x}^{'})(x_1-x_1^{'})+ \cdots + \alpha_{in}(\underline{x},\underline{x}^{'})(x_n-x^{'}_n), \ for \ i=1,\cdots,p.
\]

If $f(\underline{x})=f(\underline{x}^{'})$ and $\underline{x} \neq \underline{x}^{'}$, then every $n \times n$ minor of the matrix $\alpha=(\alpha_{ij})$ must vanish at $(\underline{x},\underline{x}^{'})$. We denote by $\mathcal{R}(\alpha)$ the ideal in $\mathcal{O}_{\mathbb{C}^{p}}$ generated by the $n\times n$ minors of $\alpha$. Then we define the \textit{double point space $D^2(f)$} (as a complex space) by
\[
D^{2}(f)=V((f\times f)^{\ast}\mathcal{I}_{p}+\mathcal{R}(\alpha)).
\]

Although the ideal $\mathcal{R}(\alpha)$ depends on the choice of the coordinate functions of $f$, in \cite{[8]} it is proved that $D^{2}(f)$ is well-defined. It is easy to see that the points in the underlying set of $D^{2}(f)$ are exactly the ones of type $(\underline{x},\underline{x}^{'})$ with $\underline{x} \neq \underline{x}^{'}$, $f(\underline{x})=f(\underline{x}^{'})$ and the ones of type $(\underline{x},\underline{x})$ such that $\underline{x}$ is a singular point of $f$.

Let $f:(\mathbb{C}^n,0)\rightarrow(\mathbb{C}^{p},0)$ be a finite map germ and take a representative of $f$ defined on a small enough open neighbourhood of the origin. Denote by $I_{p}$ and $R(\alpha)$ the stalks at $0$ of $\mathcal{I}_{p}$ and $\mathcal{R}(\alpha)$. We define \textit{the double point space of the map germ $f$} as the complex space germ 

\[
(D^{2}(f),0)=(V((f \times f)^{\ast}I_{p}+R(\alpha)),0)
\]

\begin{remark}
    In the special case where $p=n+1$ and $f$ has corank 1, we can present equations for $D^2(f)$. Let $f:(\mathbb{C}^n,0)\rightarrow(\mathbb{C}^{n+1},0)$ be a finite corank 1 map germ. In this case, we can write $f$ in a normal form $$f(\textbf{x},y)=(\textbf{x},p(\textbf{x},y),q(\textbf{x},y)),$$ where $\textbf{x}=(x_1,\dots,x_{n-1})\in\mathbb{C}^{n-1},\,y\in\mathbb{C},$ for some $p,q:(\mathbb{C}^n,0)\rightarrow(\mathbb{C},0).$ In this setting, one can compute explicit equations for $D^2(f)$ via the method of the divided differences as in following: define $$\widetilde{p}(\textbf{x},y,z)=\dfrac{p(\textbf{x},y)-p(\textbf{x},z)}{y-z}\ \ \ \text{and}\ \ \ \widetilde{q}(\textbf{x},y,z)=\dfrac{q(\textbf{x},y)-q(\textbf{x},z)}{y-z}.$$ These are well-defined functions and we have that $$(D^2(f),0)=(V(\widetilde{p},\widetilde{q}),0).$$
\end{remark}

Now we consider the dimension $p=n+1$ case. Another important space to study the topology of $f(\mathbb{C}^{n})$ is the double point space $D(f)$ in the source. According to the literature (see \cite{[12]}), an appropriate analytic structure for this space is the one given by Fitting ideals. In what follows, $f_{\ast}\mathcal{O}_n$ denotes $\mathcal{O}_n$ considered as an $\mathcal{O}_{n+1}$-module, via composition with $f$, and $Fitt_k(f_{\ast}\mathcal{O}_n)$ is its $k$th-Fitting ideal. The analytic structure of this space is defined more precisely in the following definition:

\begin{definition} Let $U \subset \mathbb{C}^n$ and $V \subset \mathbb{C}^{n+1}$ be open sets. Suppose that the map $f:U\rightarrow V$ is finite, that is, holomorphic, closed and finite-to-one. Let ${\pi}|_{D^2(f)}:D^2(f) \subset U \times U \rightarrow U$ be the restriction to $D^2(f)$ of the projection $\pi$ which is the projection onto the first factor. The \textit{double point space} is the complex space $$D(f)=V(Fitt_0({\pi}_{\ast}\mathcal{O}_{D^2(f)})).$$ Set theoretically we have the equality $D(f)=\pi(D^{2}(f))$.
\end{definition}

\begin{remark}\label{remarkdimdf}(a) If $f:U \subset \mathbb{C}^n \rightarrow V \subset \mathbb{C}^{n+1} $ is finite and generically one-to-one, then $D^2(f)$ is Cohen-Macaulay and has dimension $n-1$ \rm(\textit{see} \rm\cite[\textit{Prop.} \rm 2.1]{[11]})\textit{. In particular, $(D(f),0)$ is a germ of hypersurface in $(\mathbb{C}^n,0),$ the source of $f$.}\\

\noindent\textit{(b) In this work, except for $k=2$, we only study the multiple point spaces $D^k(f)$ in the corank 1 case. For definitions of these spaces for $k>2$ in corank 1 case, see \rm{\cite[\textit{Sec.} 9.5]{[7]}}}.
\end{remark}

\subsection{Finite determinacy, crosscaps and triple points}

$ \ \ \ \ $ The notion of a finitely determined map germ from $(\mathbb{C}^n,0)$ to $(\mathbb{C}^{n+1},0)$ is central in this work. To that end, we present the classical definition of this notion. In the particular case where $n=2,$ the number of crosscaps and triple points that appear in a stabilization of a finitely determined map germ from $(\mathbb{C}^2,0)$ to $(\mathbb{C}^3,0)$ also plays a very important role in proving the first main result of this work, as we will see in the proof of Theorem \ref{mainresult}.

\begin{definition}(a) Two map germs $f,g:(\mathbb{C}^n,0)\rightarrow (\mathbb{C}^{n+1},0)$ are $\mathcal{A}$-equivalent, denoted by $g\sim_{\mathcal{A}}f$, if there exist germs of diffeomorphisms $\Phi:(\mathbb{C}^n,0)\rightarrow (\mathbb{C}^n,0)$ and $\Psi:(\mathbb{C}^{n+1},0)\rightarrow (\mathbb{C}^{n+1},0)$, such that $g=\Psi \circ f \circ \Phi$.\\

 \noindent(b) A map germ $f:(\mathbb{C}^n,0) \rightarrow (\mathbb{C}^{n+1},0)$ is $\mathcal{A}$-finitely determined (or ``finitely determined'' for simplicity) if there exists a positive integer $k$ such that for any $g$ with $k$-jets satisfying $j^kg(0)=j^kf(0)$ we have $g \sim_{\mathcal{A}}f$.
\end{definition}

\begin{remark}\label{remarktriplepoints} (a) Consider a finite map germ $f:(\mathbb{C}^n,0)\rightarrow (\mathbb{C}^{n+1},0)$. By Mather-Gaffney criterion \rm(\cite[\textit{Th}. \rm 2.1]{Wall})\textit{, $f$ is finitely determined if and only if there is a finite representative $f:U \rightarrow V$, where $U\subset \mathbb{C}^n$, $V \subset \mathbb{C}^{n+1}$ are open neighbourhoods of the origin, such that $f^{-1}(0)=\lbrace 0 \rbrace$ and the restriction $f:U \setminus \lbrace 0 \rbrace \rightarrow V \setminus \lbrace 0 \rbrace$ is stable.} \\

\noindent\textit{(b) When $n=2$, item (a) implies that the only singularities of $f$ on $U \setminus \lbrace 0 \rbrace$ are crosscaps (or Whitney umbrellas), transverse double points and triple points. By shrinking $U$ if necessary, we can assume that there are no crosscaps or triple points in $U$. Then, since we are in the nice dimensions of Mather }\rm(\cite[\textit{p}. \rm 208]{Mather})\textit{, we can take a stabilization $\mathcal{F}$ of $f$}, 

\begin{center}
$\mathcal{F}:U \times T \rightarrow \mathbb{C}^{4}$, $\mathcal{F}(x,y,t)=(f_{t}(x,y),t)$, 
\end{center}

\noindent \textit{where $T$ is a neighbourhood of $0$ in $\mathbb{C}$. It is well known that the number of cross-caps of $f_t$ (denoted by $C(f)$) and the number of triple points of $f_t$ (denoted by $T(f)$) are independent of the particular choice of the stabilization (see for instance} \rm{\cite{mond7}}\textit{). These are analytic invariants of $f$ and they can be computed as follows (see }\rm{\cite{[8]}}\textit{):}
\[
  C(f)=dim_{\mathbb{C}}\frac{\mathcal{O}_2}{Rf} \ \ \ \ \ \  \quad T(f)=dim_{\mathbb{C}}\frac{\mathcal{O}_3}{Fitt_2(f_{\ast}\mathcal{O}_2)}
\]
\noindent \textit{where $Rf$ is the ideal generated by the maximal minors of the jacobian matrix of $f$, called the ramification ideal of $f$.}

\end{remark}

We remark that the space $D(f)$ plays a fundamental role in the study of the finite determinacy. In \cite[Th. 2.14]{MararMond}, Marar and Mond presented necessary and sufficient conditions for a corank $1$ map germ from $(\mathbb{C}^n,0)$ to $(\mathbb{C}^p,0)$ to be finitely determined in terms of properties of $D^2(f)$ and other multiple points spaces. Also in $n=2$ case, Marar, Nu\~{n}o-Ballesteros and Pe\~{n}afort-Sanchis extended this criterion of finite determinacy to the corank $2$ case (\cite[Th. 3.4 and Cor. 3.5]{[11]}). They proved the following result.

\begin{theorem}\rm(\cite{MararMond}, \cite{[11]})\label{criterio} \textit{
Let $f:(\mathbb{C}^2,0)\rightarrow(\mathbb{C}^{3},0)$ be a finite and generically $1$ $-$ $1$ map germ. Then $f$ is finitely determined if and only if the Milnor number of $D(f)$ at $0$ is finite.}
\end{theorem}

\subsection{Quasihomogeneous map germs from $(\mathbb{C}^n,0)$ to $(\mathbb{C}^{n+1},0)$}

$ \ \ \ \  $ In this work, we also aim to study quasihomogeneous map germs. Thus, it is convenient to present a precise definition of this kind of map.

\begin{definition}
    A polynomial $p(x_1,\dots,x_n)$ is said to be quasihomogeneous if there exists positive integers $w_1,\dots,w_n$ with no common factor and a positive integer $d$ such that $$p(k^{w_1}x_1,\dots,k^{w_n}x_n)=k^dp(x_1,\dots,x_n).$$ The integer $w_i$ is called the weight of the variable $x_i$ and $d$  is called the weighted degree of $p.$ In this case, we say that $p$ is of type $(d;w_1,\dots,w_n).$ We say that $f:(\mathbb{C}^n,0)\rightarrow(\mathbb{C}^p,0)$  is a quasihomogeneous map germ of type $(d_1,\dots,d_p;w_1,\dots,w_n)$ if each coordinate function $f_i$ is quasihomogeneous of type $(d_i;w_1,\dots,w_n).$ In particular, when $(n,p)=(2,3),$ we have that $f$ is quasihomogeneous of type $(d_1,d_2,d_3;w_1,w_2).$
\end{definition}

Note that if $f:(\mathbb{C}^n,0)\rightarrow (\mathbb{C}^{n+1},0)$ is a corank $1$ map germ then we can write $f$ in the normal form $$f(x_1,\dots,x_n)=(x_1,\dots,x_{n-1},p(x_1,\dots,x_n),q(x_1,\dots,x_n)).$$ If in addition $n=2$ and $f$ is quasihomogeneous map germ, then we can write $p(x,y)=\theta_1 y^n +x\tilde{p}(x,y)$ and $q(x,y)=\theta_2y^m+x\tilde{q}(x,y)$, for some $n,m \in \mathbb{N}$, $\theta_i \in \mathbb{C}$ and $\tilde{p},\tilde{q} \in \mathcal{O}_2$ with $\tilde{p}(x,0)=\tilde{q}(x,0)=0$. This is explained more precisely in the following lemma, where a normal form for $f$ which is more convenient for our purposes is presented.

\begin{lemma}\rm(\cite[\textit{Lemma} \rm 2.11]{normalform})\label{lemma corank 1} \textit{Let $g(x,y)=(g_1(x,y),g_2(x,y),g_3(x,y))$ be a finitely determined, quasihomogeneous, corank 1 map germ of type $(d_1,d_2,d_3; w_1,w_2)$. Then $g$ is $\mathcal{A}$-equivalent to a quasihomogeneous map germ $f$ with type $(d_{i_1}=w_1,d_{i_2},d_{i_3};w_1,w_2)$, which is written in the form}

\begin{equation}\label{eq14} 
f(x,y)=(x, y^n+xp(x,y), \beta y^m+ xq(x,y)),
\end{equation}

\noindent \textit{for some integers $n,m\geq 2$, $\beta \in \mathbb{C}$, $p,q \in \mathcal{O}_2$, $p(x,0)=q(x,0)=0$, where $(d_{i_1},d_{i_2},d_{i_3})$ is a permutation of $(d_1,d_2,d_3)$ such that $d_{i_2}\leq d_{i_3}$.}
\end{lemma}

In Section \ref{sec4}, we will present a quasihomogeneous normal form for finitely determined, quasihomogeneous, corank 1 map germs from $(\mathbb{C}^n,0)$ to $(\mathbb{C}^{n+1},0),$ with $n\ge3$ (see Lemma \ref{forma}), which is a generalization of Lemma \ref{lemma corank 1}. Returning to the $n=2$ case, in \cite{form}, Mond showed that if $f$ is finitely determined and quasihomogeneous then the invariants $C(f)$, $T(f)$ and also the $\mathcal{A}_e-$codimension of $f$ (denoted by $\mathcal{A}_e$-codim$(f)$) are determined by the weights and the degrees of $f$ (see \cite{[7]} for the definiton of $\mathcal{A}_e-$codimension of $f$). More precisely, he showed the following result.

\begin{theorem}\label{mondformulas} \rm({\cite{form}}) \textit{Let $f:(\mathbb{C}^2,0)\rightarrow (\mathbb{C}^3,0)$ be a finitely determined, quasihomogeneous map germ of type $(d_1,d_2,d_3;w_1,w_2)$. Then}

\begin{center}
$C(f)=\dfrac{1}{w_1w_2}\bigg((d_2-w_1)(d_3-w_2)+(d_1-w_2)(d_3-w_2)+(d_1-w_1)(d_2-w_1)\bigg)$, $ \ \ \ $  $ \ \ \ $ $T(f)=\dfrac{1}{6w_1w_2}(\delta-\epsilon)(\delta-2\epsilon)+\dfrac{C(f)}{3}$ $ \ \ \ $ and $ \ \ \ $ $\mathcal{A}_e-\text{codim}(f)=\dfrac{1}{2}\left(\mu(D(f))-4T(f)+C(f)-1 \right)$
\end{center}

\noindent \textit{where $\epsilon = d_{1}+d_{2}+d_{3}-w_1-w_2$ and $\delta=d_{1}d_{2}d_{3}/(w_1w_2)$.}

\end{theorem}

    \section{The multiplicity of the image of a finite map germ for the case $n=2$}\label{multiplicidade}

$ \ \ \ \ $ In Introduction we give the definition of the multiplicity of a germ of hypersurface $(\textbf{X},0)=(V(F),0)$ in terms of the multiplicity of $F:(\mathbb{C}^n,0)\rightarrow (\mathbb{C},0)$. When $(\textbf{X},0)$ is the image of a generically one-to-one map germ $f:(\mathbb{C}^n,0)\rightarrow (\mathbb{C}^{n+1},0)$, another way to calculate the multiplicity of $(\textbf{X},0)$ is using a linear generic projection from $(\mathbb{C}^{n+1},0)$ on $(\mathbb{C}^n,0)$ restricted to $(\textbf{X},0)$. This is explained in the following remark.

\begin{remark}\label{remarkMulti} Let $(X,0)$ and $(Y,0)$ be germs of irreducible analytic sets. Let $f:(X,0)\rightarrow (Y,0)$ be a finite surjective analytic map germ. We denote the degree of a map $f$ by $deg(f)$. Roughly speaking, the degree of $f$ is the number of pre-images of a generic value in the image of $f$. The precise definition is given for instance in \rm{\cite[\textit{Def.} D.2]{[7]}}\textit{. An important fact about the concept of the degree of a map is its multiplicative property. Suppose $g:(Y,0)\rightarrow(W,0)$ is a finite surjective analytic map germ, where $(W,0)$ is an irreducible analytic set, then $deg(g\circ f)=deg(g)\cdot deg(f)$.}
	
 \textit{Suppose that $(X,0)$ is of dimension $n$ and $(X,0)\subset (\mathbb{C}^{n+1},0)$. Let $l_1,l_2,\cdots, l_n:(\mathbb{C}^{n+1},0)\rightarrow (\mathbb{C},0)$ be generic (reduced) linear forms. Let $\pi:(X,0)\rightarrow (\mathbb{C}^n,0)$, $\pi=(l_1,\cdots, l_n)$, be the restriction to $X$ of the (generic) linear projection $\pi$ from $\mathbb{C}^{n+1}$ to $\mathbb{C}^n$. Here, ``\textit{generic linear projection}'' means that $Ker(\pi):= \pi^{-1}(0)$ is a generic line in $\mathbb{C}^{n+1}$ such that $Ker(\pi) \cap X = \{0\}$.}
 
  \textit{For a generic $x$ close enough to $0$, $\pi^{-1}(x)$ is a subspace parallel to $Ker(\pi)$ which intersects $X$ in a finite number of points; this number is precisely $m(X,0)$ (see for instance} \rm{\cite[\textit{Sec.} D.3]{[7]}}\textit{). In other words, the multiplicity can be seen as the local intersection number at $0$ of $X$ with a generic line in $\mathbb{C}^{n+1}$. We note that this local intersection number is independent of the choice of the generic line (see} \rm{\cite[\textit{Sec.} D.3]{[7]}}\textit{).}

\textit{Let $f:(\mathbb{C}^n,0)\rightarrow(\mathbb{C}^{n+1},0)$ be a  finite map germ and of degree one onto its image $(\textbf{X},0):=(f(\mathbb{C}^n),0).$ Let $\pi:(\mathbb{C}^{n+1},0)\rightarrow(\mathbb{C}^n,0), \pi(X_1,\dots,X_{n+1})=(l_1(X_1,\dots,X_{n+1}),\dots,l_n(X_1,\dots,X_{n+1})),$ be a generic linear projection. Since the degree of $f$ over its image is 1, it follows that $$m(\textbf{X},0)=deg(\pi|_\textbf{X})=deg(\pi|_\textbf{X})\cdot deg(f)=deg(\pi\circ f).$$ Furthermore, we have that $$m(\textbf{X},0)=deg(\pi \circ f)= dim_{\mathbb{C}}\dfrac{\mathcal{O}_n}{\langle l_1 \circ f,\dots, l_n \circ f\rangle}.$$ In addition, assuming that $n=2$ and $f$ is a finite corank $1$ map germ and write $f$ in the normal form} 

\begin{center}
$f(x,y)=(x,\alpha y^{m_1}+xp(x,y), \beta y^{m_2}+xq(x,y))$. 
\end{center}

\noindent \textit{Suppose that $m_1<m_2$ and $\alpha \neq 0$. It follows that}

\begin{center}
$m(\textbf{X},0)=deg(\pi \circ f)= dim_{\mathbb{C}}\dfrac{\mathcal{O}_2}{\langle l_1 \circ f, l_2 \circ f\rangle}=m_1$.
\end{center}

\end{remark}

Now we consider only the $n=2$ case. Suppose now that $(\textbf{X},0)$ is a germ of surface in $(\mathbb{C}^3,0)$ given as the zero set of a quasihomogeneous map germ $F:(\mathbb{C}^3,0)\rightarrow (\mathbb{C},0)$. Suppose that the singular set of $(\textbf{X},0)$ has dimension one and also that $(\textbf{X},0)$ has a smooth normalization $f:(\mathbb{C}^2,0) \rightarrow (\mathbb{C}^3,0)$. Since $F$ is quasihomogeneous, we have that $f$ is quasihomogeneous. Suppose that is a quasihomogeneous map germ of type $(d_1,d_2,d_3;w_1,w_2)$ in a suitable system of coordinates for some positive integers $d_1,d_2,d_3,w_1$ and $w_2$. Note that $f$ is in particular a finite map germ of degree one onto its image, hence we obtain from \cite[Prop. 1.15]{form} that $F$ is of type $(d; d_1,d_2,d_3)$ where the weighted degree $d$ of $F$ is equal to $(d_1d_2d_3)/(w_1w_2)$. At this point, let's return to Question $1$ in Introduction, i.e., in this setting is the multiplicity of $(\textbf{X},0)$ determined by the weights and degrees of $f$?

Example \ref{exquasihomog} shows that in general the answer is negative. However, if we assume that $f$ is finitely determined, then Question $1$ has a positive answer, as we will see in the following lemma.

\begin{lemma}\label{lema aux5} Let $f=(f_1,f_2,f_3)$ be a singular quasihomogeneous map germ from $(\mathbb{C}^2,0)$ to $(\mathbb{C}^3,0)$ of type $(d_1,d_2,d_3; w_1,w_2)$. Without loss of generality, suppose that $1\leq d_1 \leq d_2 \leq d_3$, where the weighted degree of $f_i$ is $d_i$. Suppose that $f$ is finitely determined, then the multiplicity of the image of $f$ is determined by the weights and degrees of $f$. More precisely:

\begin{center}
$m(f(\mathbb{C}^2),0)=\dfrac{d_1 d_2}{w_1w_2}$.
\end{center}

\end{lemma}

\begin{proof} Note that the Lemma holds for the crosscap $g(x,y)=(x,y^2,xy).$ Suppose that $f$ is not $\mathcal{A}$-equivalent to the crosscap. Since $f_i$ is quasihomogeneous, for $i=1,2$ and $3$ we can write  
\begin{equation}\label{eq15}
f_i(x,y)= \displaystyle { x^{e_i}y^{l_i}\prod_{j=1}^{k_i}}(x^{w_2}- \gamma_{i,j} y^{w_1})
\end{equation}

\noindent where $\gamma_{i,j} \neq 0$ for all $i$ and $j$ and $k_i,e_i,l_i\geq 0$. \\

\noindent\textbf{Statement 1:} $V(f_1)\cap V(f_2)=\left\{(0,0)\right\}.$\\

\noindent\textit{Proof of the Statement 1.} Suppose that $V(f_1) \cap V(f_2) \neq \lbrace (0,0) \rbrace$. Thus, this intersection contains (as a subset) either $V(x)$ or $V(y)$ or even $V(x^{w_2}- \gamma_{1,j} y^{w_1})$ for some $j$. In particular, since $f$ is finite, we have that $V(f_1)\cap V(f_2)\cap V(f_3)=\left\{(0,0)\right\}.$ Therefore, the defining equation of a curve in the intersection $V(f_1)\cap V(f_2)$ is not a factor of $f_3.$ Let us divide the proof of Statement 1 in cases.\\

\noindent \textbf{Case A:} Suppose that $V(x^{w_2}- \gamma_{1,j} y^{w_1}) \subset V(f_1) \cap V(f_2)$ for some $j$.\\

Consider a parametrization $\varphi_j:(\mathbb{C},0)\rightarrow (\mathbb{C}^2,0)$ for $V(x^{w_2}- \gamma_{1,j} y^{w_1})$ given by $\varphi_j(u)=(u^{w_1},\theta_j u^{w_2})$, where $\theta_j$ is a non-zero constant. Note that the composition $f \circ \varphi_j$ is finite and generically $d_3$-to-$1$. Since $f$ is finitely determined it follows that $d_3 \in \lbrace 1,2 \rbrace$.  
Since $d_1 \leq d_2 \leq d_3$ it follows that $d_i \in \lbrace 1,2 \rbrace$ for all $i$. For each choice of degrees the only possibilities for the weights are either $(w_1,w_2)=(1,1)$ or $(w_1,w_2)=(1,2)$ (and similarly, $(w_1,w_2)=(2,1)$). Table \ref{tabela6} shows all possibilities for the weights and degrees of $f$.

\begin{table}[H]
\centering
{\def\arraystretch{2.1}\tabcolsep=30pt 

\begin{tabular}{ c | c | c | c }

\hline
\rowcolor{lightgray}
Subcase  &  $d_1$ & $d_2$ & $d_3$  \\

\hline
a.1)     & $1$ & $1$  & $1$    \\
\hline
a.2)     & $1$ & $1$  & $2$    \\
\hline
a.3)     & $1$ & $2$  & $2$    \\

\hline
a.4)     & $2$ & $2$  & $2$    \\

\hline
\end{tabular}
}
\caption{Possibilities for the weights and degrees of $f$}\label{tabela6}
\end{table}

Clearly the subcases a.1) and a.2) are not singular. One way to see this is to use Theorem \ref{mondformulas} to calculate the expected number of crosscaps of $f$. In these subcases we obtain that $C(f)$ should be equal to $0$. For the subcase a.3) we have that $f$ is either of type $(1,2,2;1,1)$ or of type $(1,2,2;1,2)$ (or even $(1,2,2;2,1)$). If $f$ is of type $(1,2,2;1,1)$ then $\mu(D(f),0)=\mu(D^2(f),0)=0$ and hence it is stable and therefore $\mathcal{A}$-equivalent to the crosscap $g(x,y)=(x,y^2,xy)$. For the other types, using the same argument used in subcases a.1) and a.2) we conclude that $f$ is not singular. For the subcase a.4) note that $f$ is of either of type $(2,2,2;1,1)$ or of type $(2,2,2;1,2)$ (or similarly $(2,2,2;2,1)$). Using Theorem \ref{mondformulas} again we obtain that $f$ is not finitely determined since the $\mathcal{A}_e$-codim$(f)$ is not an integer number. Therefore, Case A cannot occur.\\ 

\noindent \textbf{Case B:} Suppose that $V(x) \subset V(f_1) \cap V(f_2)$.\\

Consider the parametrization $\varphi:(\mathbb{C},0)\rightarrow (\mathbb{C}^2,0)$ for $V(x)$ given by $\varphi(u)=(0,u)$. Note that the composition $f \circ \varphi$ is finite and generically $d_3/w_2$-to-$1$. Since $f$ is finitely determined it follows that either $d_3/w_2=1$ or $d_3/w_2=2$. If $d_3/w_2=1$, then $f$ has corank 1. Since $d_1\le d_2\le d_3,$ then $f=(x^{e_1},x^{e_2},y),$ where $e_i\ge1.$ In this case, $C(f)$ is not finite, which contradicts the hypothesis. Suppose now that $d_3/w_2=2$. In this case, one can verify that $w_1=1$, $f_1=a_1x^{e_1}+b_1x^{e_1-w_2}y,\ f_2=a_2x^{e_2}+b_1x^{e_2-w_2}y,$ and $f_3=y^2+b_3x^{w_2}y,$ for some $a_1,a_2,b_1,b_2,b_3\in\mathbb{C}$ and $e_1,e_2\ge w_2.$ Hence, considering the possible forms to $f_i's$, we conclude that $f$ is $\mathcal{A}$-equivalent to the crosscap or $C(f)$ is not finite. Therefore, Case B cannot occur.\\

\noindent \textbf{Case C:} Suppose that $V(y) \subset V(f_1) \cap V(f_2)$.\\

The proof of this case is analogous to the proof of Case 2. Therefore, Case C cannot occur, and we conclude the proof of Statement 1.\\

Now, let $\pi:(\mathbb{C}^3,0)\rightarrow (\mathbb{C}^2,0)$, $\pi(X,Y,Z)=(\alpha_1 X+ \alpha_2 Y+ \alpha_3 Z,\beta_1 X + \beta_2 Y + \beta_3 Z)$, be a generic projection in the sense of Remark \ref{remarkMulti}. Thus 

\begin{eqnarray*}
    m(f(\mathbb{C}^2),0) &=& deg(\pi_{|_{f(\mathbb{C}^2)}} \circ f)\\
    &=& dim_{\mathbb{C}}\dfrac{\mathcal{O}_2}{\langle \ \alpha_1 f_1+ \alpha_2 f_2+ \alpha_3 f_3, \ \beta_1 f_1 + \beta_2 f_2 + \beta_3 f_3 \ \rangle}\\
    &=& dim_{\mathbb{C}}\dfrac{\mathcal{O}_2}{\langle \ \alpha_1 f_1+ \alpha_2 f_2+ \alpha_3 f_3, \  \tilde{\beta}_2 f_2 + \tilde{\beta}_3 f_3 \ \rangle}
\end{eqnarray*}

\noindent where $\tilde{\beta}_2=\beta_2-(\beta_1 \alpha_2 /\alpha_1)$ and $\tilde{\beta}_3=\beta_3-(\beta_3 \alpha_3 /\alpha_1)$. Note that $h_1 := \alpha_1 f_1+ \alpha_2 f_2+ \alpha_3 f_3$ and $h_2: = \tilde{\beta}_2 f_2 + \tilde{\beta}_3 f_3$ are semi-quasihomogeneous (where $\alpha_1 f_1$ (respectively, $\tilde{\beta}_2 f_2$) is the filtration-$0$ part of $h_1$ if $d_1<d_2$ (respectively, $h_2$ if $d_2<d_3$)). By the Statement 1, we have that $V(f_1) \cap V(f_2) = \lbrace (0,0) \rbrace$ then we can apply an appropriate version of the generalized Bézout's theorem (see, for instance, \cite{Chirka}) to obtain that 

\begin{center}
$m(f(\mathbb{C}^2),0) \ =  dim_{\mathbb{C}}\dfrac{\mathcal{O}_2}{\langle \ \alpha_1 f_1+ \alpha_2 f_2+ \alpha_3 f_3, \  \tilde{\beta}_2 f_2 + \tilde{\beta}_3 f_3 \ \rangle} = \dfrac{d_1d_2}{w_1w_2}$
\end{center}

\noindent which shows in particular that the multiplicity of the image of $f$ is determined by the weights and degrees of $f$.\end{proof}\\

The following example shows that Lemma \ref{lema aux5} does not work if we remove the hypothesis that $f$ is finitely determined.

\begin{example}\label{exquasihomog} Consider the map germ $f(x,y)=(x,y^3,xy)$, it is a homogeneous map germ of type $(1,3,2;1,1)$. Since the codimension of the ideal $\langle x,x^3+y^3,xy \rangle $ is finite we obtain that $f$ is finite. Recall that $f_{\ast}\mathcal{O}_2$ denotes $\mathcal{O}_2$ considered as an $\mathcal{O}_3$-module, via composition with $f$, and $Fitt_k(f_{\ast}\mathcal{O}_2)$ is its $k$th-Fitting ideal. In this sense, a defining equation for the image of $f$ is given by the $0$th-Fitting ideal of $f_{\ast}\mathcal{O}_2$ (see \rm{\cite{[12]}}\textit{). A presentation matrix of $f_{\ast}\mathcal{O}_2$ is given by}

\begin{center}
$  \begin{bmatrix}

Z &    -X &      0 \\
0 & Z &       -X \\
-XY &    0 & Z

\end{bmatrix}.$
\end{center}

\textit{Hence, $F(X,Y,Z)=Z^3-X^3Y$ is a defining equation of $(\textbf{X},0)$. Note that $F$ is quasihomogeneous of type $(6;1,3,2)$ and the ideal generated by $F$ in $\mathcal{O}_3$ is radical. Therefore, by} \rm{\cite[\textit{Prop.} 3.1]{[12]}} \textit{we obtain that $f$ is of degree one onto its image. On the other hand, consider the map germ $g(x,y)=(x,y^3,xy+y^2)$, it is a homogeneous map germ of type $(1,3,2;1,1)$ (the same type as $f$). Denote the image of $g$ by $(\textbf{Y},0)$. Clearly $g$ is finite. A presentation matrix of $g_{\ast}(\mathcal{O}_2)$ is given by}

\begin{center}
$  \begin{bmatrix}

Z &    -X &      -1 \\
-Y & Z &       -X \\
0 &    -Y-XZ & Z+X^2

\end{bmatrix}.$
\end{center}

\textit{Therefore, $G(X,Y,Z)=Z^3-X^3Y-Y^2-3XYZ$ is a defining equation of $(\textbf{Y},0)$. Again, it follows by} \rm{\cite[\textit{Prop.} 3.1]{[12]}} \textit{that $g$ is of degree one onto its image. We have that $g$ is quasihomogeneous of same type as $f$. However, the multiplicity of $(\textbf{X},0)$ is $3$, on the other hand the multiplicity of $(\textbf{Y},0)$ is $2$. This shows that Question $1$ has a negative answer (note that $g$ is finitely determined while $f$ is not).}
\end{example}

\begin{example}
    Consider the map germs $f,g:(\mathbb{C}^2,0)\rightarrow(\mathbb{C}^3,0)$ given by $f(x,y)=(x^2,y^3,(x+y)^5)$ and $g(x,y)=(x^2,y^3,(x+y)^7).$ We have that $f$ and $g$ are finitely determined (see \rm{\cite[\textit{Ex.} 16]{Guille}}\textit{). Furthermore, $$m(f(\mathbb{C}^2),0)=6=m(g(\mathbb{C}^2),0).$$ However, note that the degree $d_3$ of $f$ is $5,$ while the degree $d_3$ of $g$ is $7.$ This example shows that the multiplicity of the image of a finitely determined, quasihomogeneous map germ does not depend of the degree $d_3$ (provided that the degrees are ordered and $d_3$ is the largest of them).}
\end{example}

The reader may wonder about the possibility of extending the Lemma \ref{lema aux5} for finitely determined, quasihomogeneous $f:(\mathbb{C}^n,0)\rightarrow(\mathbb{C}^{n+1},0)$ map germs (as in Question 2 in Introduction). We will address this problem in Section \ref{sec4} with an additional assumption (see Theorem \ref{lemamult}).

\section{The Zariski's multiplicity conjecture for surfaces}\label{secZariski}

$ \ \ \ \ $ This section is devoted to proving our first main result, Theorem \ref{mainresult} (first stated in the Introduction). As a convenience to the reader, we now repeat Theorem \ref{mainresult}. The following statement is verbatim from the Introduction.

\begin{theorem}\label{Main1}
    Let $(\textbf{X},0)$ and $(\textbf{Y},0)$ be germs of irreducible surfaces in $(\mathbb{C}^3,0)$ defined as the zero set of reduced quasihomogeneous map germs $F,G:(\mathbb{C}^3,0)\rightarrow (\mathbb{C},0)$, respectively. Suppose that $(\textbf{X},0)$ (respectively $(\textbf{Y},0)$) has a smooth normalization and outside the origin the only singularities of $(\textbf{X},0)$ (respectively $(\textbf{Y},0)$) are transverse double points. If $(\textbf{X},0)$ and $(\textbf{Y},0)$ are topologically equivalent, then $m(\textbf{X},0)=m(\textbf{Y},0)$.
\end{theorem} 
 
\begin{proof} Let $f:(\mathbb{C}^2,0)\rightarrow (\textbf{X},0) \subset (\mathbb{C}^3,0)$ (respectively $g:(\mathbb{C}^2,0)\rightarrow (\textbf{Y},0) \subset (\mathbb{C}^3,0)$) be a smooth normalization of $(\textbf{X},0)$ (respectively $(\textbf{Y},0)$). Since the only singularities of $(\textbf{X},0)$ (respectively $(\textbf{Y},0)$) outside the origin are transverse double points, by geometric criterion of Mather-Gaffney \cite[Th. \rm 2.1]{Wall} we obtain that $f$ and $g$ are finitely determined map germs. Since $F$ and $G$ are quasihomogeneous, $f$ and $g$ can be taken to be quasihomogeneous map germs. Hence, we are in the context of the results of the previous section. Suppose that $f$ (respectively $g$) is of type $(d_1,d_2,d_3; w_1,w_2)$ (respectively $(\tilde{d}_1,\tilde{d}_2,\tilde{d}_3; \tilde{w}_1, \tilde{w}_2))$.\\ 

\noindent \textbf{Statement 1:} The weights of $f$ and $g$ are the same.\\

\noindent\textit{Proof of the Statement 1.} By hypothesis, $(\textbf{X},0)$ and $(\textbf{Y},0)$ are topologically equivalent. Hence, there is a germ of homeomorphism $h:(\mathbb{C}^3,0)\rightarrow (\mathbb{C}^3,0)$ such that $h(\textbf{X},0)=(\textbf{Y},0)$. By \cite[Lemma 5]{lev}, we can extend $h$ at the normalization, i.e., there exists a germ of homeomorphism $\tilde{h}:(\mathbb{C}^2,0)\rightarrow (\mathbb{C}^2,0)$ such that the following diagram

\begin{figure}[H]
        \begin{center}
        \begin{tikzpicture}
            \draw (0,0) node {$(\mathbb{C}^2,0)$};
            \draw (-1.4,-2.5) node {$(\textbf{X},0)\subset$};
            \draw (-1.7,0) node {$(D(f),0)\subset$};
            \draw (4.9,-2.5) node {$\supset(\textbf{Y},0)$};
            \draw (5.2,0) node {$\supset(D(g),0)$};
            \draw[->] (1,0) to (2.5,0);
            \draw (3.5,0) node {$(\mathbb{C}^2,0)$};
            \draw[->] (3.5,-0.5) to (3.5,-2);
            \draw (3.5,-2.5) node {$(\mathbb{C}^3,0)$};
            \draw[->] (1,-2.5) to (2.5,-2.5);
            \draw (0,-2.5) node {$(\mathbb{C}^3,0)$};
            \draw[->] (0,-0.5) to (0,-2);
            \draw (1.75,0.5) node {$\widetilde{h}$};
            \draw (1.75,-2) node {$h$};
            \draw (-0.5,-1.25) node {$f$};
            \draw (4,-1.25) node {$g$};
        \end{tikzpicture}
        \end{center}
        \end{figure}
 
\noindent commutes. Furthermore, $\tilde{h}(D(f))=D(g)$, where $D(f)$ and $D(g)$ are (reduced) curves by Theorem \ref{criterio}. Hence, we have that $D(f)$ and $D(g)$ are topologically equivalent. Write $D(f)=V(\lambda(x,y))$ and $D(g)=V(\tilde{\lambda}(x,y))$ where $\lambda,\tilde{\lambda}:(\mathbb{C}^2,0)\rightarrow (\mathbb{C},0)$. Since $f$ and $g$ are quasihomogeneous, we have that $\lambda$ and $\tilde{\lambda}$ are both quasihomogeneous. Hence, we can write

\begin{equation}\label{eq8}
\lambda(x,y)=\displaystyle { x^{s}y^{l}\prod_{i=1}^{k}}(y^{w_1}-\alpha_i x^{w_2}) \ \ \ \ \ \ \  \tilde{\lambda}(x,y)=\displaystyle { x^{\tilde{s}}y^{\tilde{l}}\prod_{i=1}^{\tilde{k}}}(y^{\tilde{w}_1}-\tilde{\alpha_i} x^{\tilde{w}_2}).
\end{equation}

\noindent where $s,l,\tilde{s},\tilde{l} \in \lbrace 0,1 \rbrace$ and $\alpha_i, \tilde{\alpha_i} \neq 0$ for all $i$, we have also that $k,\tilde{k}\geq 0$. Suppose firstly that $D(f)$ is smooth. It follows that $D^2(f)$ is smooth. Since the Milnor numbers of $D(f)$ and $D^2(f)$ are both equal to $0$, we obtain by \cite[Th. 4.3]{[11]} (see also \cite[Th. 3.4]{MararMond} for corank $1$ analogue) that $T(f)=0$. Hence $f$ has multiplicity $2$ and \cite{Navarro} we obtain that $m(\textbf{X},0)=m(\textbf{Y},0)=2$. Therefore we can assume that $D(f)$ and $D(g)$ are singular curves throughout the rest of the proof. Now, suppose that $k=\tilde{k}=0$. Hence, $D(f)=V(xy)$, i.e., a singularity of Morse type. Note that $D^2(f)$ is a pair of distinct lines in $\mathbb{C}^2 \times \mathbb{C}^2$, thus the Milnor number of $D^2(f)$ is $1$. Again, by \cite[Th. 4.3]{[11]} we obtain that $f$ has multiplicity $2$. Therefore, we can also assume that $k,\tilde{k}\geq 1$ throughout the rest of the proof. Since $\lambda, \tilde{\lambda}$ are both quasihomogeneous having an isolated singularity at the origin and $D(f)$ is topologically equivalent to $D(g)$, by \cite{Yoshinaga} (see also \cite{Nishimura}) we obtain that the weights of $\lambda$ and $\tilde{\lambda}$ are the same. In particular, this implies that the weights of $f$ and $g$ are the same, namely, $w(x)=w_1$ and $w(y)=w_2$, which proves Statement 1.\\

Assuming the Statement 1, we can suppose that $g$ is of type $(\tilde{d}_1,\tilde{d}_2, \tilde{d}_3; w_1,w_2)$ with $ 1 \leq \tilde{d}_1 \leq \tilde{d}_2 \leq \tilde{d}_3$ (and also that $ 1 \leq d_1 \leq d_2 \leq d_3$ for the degrees of $f$). Now in order to prove that the multiplicities of $(\textbf{X},0)$ and $(\textbf{Y},0)$ are equal, by Lemma \ref{lema aux5} it is sufficient to show that all degrees of $f$ and $g$ are the same, i.e., $d_i = \tilde{d}_i$ for $i=1,2$ and $3$.\\

\noindent \textbf{Statement 2:} The degrees of $f$ and $g$ are the same.\\

\noindent\textit{Proof of Statement 2.} Denote by $d$ (respectively, $\tilde{d}$) the weighted degree of $\lambda(x,y)$ (respectively, $\tilde{\lambda}(x,y)$). Let's first show that $d=\tilde{d}$. From (\ref{eq8}), we obtain that $d=kw_1w_2+sw_1+lw_2$ and $\tilde{d}=\tilde{k}w_1w_2+\tilde{s}w_1+\tilde{l}w_2$, recall that $s,l,\tilde{s},\tilde{l} \in \lbrace 0,1 \rbrace$. Recall that  since we assume $k\geq 1$, for $i=1,\cdots,k$, we set $D(f)^i=V(\alpha_i x^{w_2}- \beta_i y^{w_1})$ the components of $D(f)$. Table \ref{tabela4} shows the values for the intersection multiplicities between (eventual) branches of $D(f)$ in terms of the weights of $\lambda$.

\begin{table}[H]
\centering
{\def\arraystretch{2.3}\tabcolsep=30pt 

\begin{tabular}{ c | c | c | c }

\hline
\rowcolor{lightgray}
Intersection multiplicity  &  $D(f)^i $ & $V(y^l)$ & $V(x^s)$ \\

\hline
$D(f)^j$     & $w_1w_2$, $ \ $ if $i\neq j$ & $lw_2$  & $sw_1$    \\
\hline
$V(x^s)$    & $sw_1$ & $sl$ & $-$ \\
\hline

$V(y^l)$ & $lw_2$ & $-$ & $-$\\

\hline
\end{tabular}
}
\caption{Intersection multiplicities of (eventual) branches of $D(f)$.}\label{tabela4}
\end{table}

Recall that we also suppose that $\tilde{k} \geq 1$, for $i=1,\cdots,\tilde{k}$, we set $D(g)^i=V( \tilde{\alpha}_i x^{w_2}- \tilde{\beta}_i y^{w_1})$ the components of $D(g)$. Table \ref{tabela5} shows the values for the intersection multiplicities between (eventual) branches of $D(g)$ in terms of the weights of $\tilde{\lambda}$.

\begin{table}[H]
\centering
{\def\arraystretch{2.3}\tabcolsep=30pt 

\begin{tabular}{ c | c | c | c }

\hline
\rowcolor{lightgray}
Intersection multiplicity  &  $D(g)^i $ & $V(y^{\tilde{l}})$ & $V(x^{\tilde{s}})$ \\

\hline
$D(g)^j$     & $w_1w_2$, $ \ $ if $i\neq j$ & $\tilde{l}w_2$  & $\tilde{s}w_1$    \\
\hline
$V(x^{\tilde{s}})$    & $\tilde{s}w_1$ & $\tilde{s}\tilde{l}$ & $-$ \\
\hline

$V(y^{\tilde{l}})$ & $\tilde{l}w_2$ & $-$ & $-$\\

\hline
\end{tabular}
}
\caption{Intersection multiplicities of (eventual) branches of $D(g)$.}\label{tabela5}
\end{table}

Since the intersection multiplicities are (embedded) topological invariants for plane curves, the values in Tables \ref{tabela4} and \ref{tabela5} must be the same (after possible reordering). In particular, we obtain that $s=1$ if and only if $\tilde{s}=1$, and $l=1$ if and only if $\tilde{l}=1$. Hence, we conclude that $d=\tilde{d}$ and therefore by \cite[Prop. 1.15]{form} we obtain that

\begin{equation}\label{eq13}
\dfrac{d_1d_2d_3}{w_1w_2}-d_1-d_2-d_3+w_1+w_2 \ =  \ \dfrac{\tilde{d}_1 \tilde{d}_2 \tilde{d}_3}{w_1w_2}-\tilde{d}_1-\tilde{d}_2-\tilde{d}_3+w_1+w_2.
\end{equation}

By \cite{Edson} we have the number of crosscaps and triple points are topological invariants for a pair of finitely determined map germs from $(\mathbb{C}^2,0)$ to $(\mathbb{C}^3,0)$. Thus we obtain that $C(f)=C(g)$ and also that $T(f)=T(g)$. By expression (\ref{eq13}) and the formulas for the invariants $C$ and $T$ described in Theorem \ref{mondformulas} we obtain the following system of equations:
\[
    \left\{
                \begin{array}{ll}

       \delta - \epsilon \ =  \tilde{\delta}-\tilde{\epsilon} \vspace{0.1cm}  \\    
       
       d_1d_2+d_1d_3+d_2d_3-(w_1+w_2)(d_1+d_2+d_3)= \tilde{d}_1\tilde{d}_2+\tilde{d}_1\tilde{d}_3+\tilde{d}_2\tilde{d}_3 -(w_1+w_2)(\tilde{d}_1+\tilde{d}_2+\tilde{d}_3) \vspace{0.2cm} \\

                  (\delta-\epsilon)(\delta-2\epsilon)=(\tilde{\delta}- \tilde{\epsilon})(\tilde{\delta}-2\tilde{\epsilon})\\
                 
                \end{array}
              \right.
  \]


\noindent where $\epsilon = d_{1}+d_{2}+d_{3}-w_1-w_2$, $\delta=d_{1}d_{2}d_{3}/(w_1w_2)$, $\tilde{\epsilon} = \tilde{d}_{1}+\tilde{d}_{2}+\tilde{d}_{3}-w_1-w_2$ and $\tilde{\delta}=\tilde{d}_{1}\tilde{d}_{2}\tilde{d}_{3}/(w_1w_2)$. Note that $d=\delta - \epsilon$, in particular $\delta - \epsilon \neq 0$. Hence, the above system is equivalent to the following system:

\[
    \left\{
                \begin{array}{ll}
                  
                  d_1d_2d_3=\tilde{d}_1\tilde{d}_2\tilde{d}_3 \vspace{0.2cm}\\
                  d_1d_2+d_1d_3+d_2d_3= \tilde{d}_1\tilde{d}_2+\tilde{d}_1\tilde{d}_3+\tilde{d}_2\tilde{d}_3 \vspace{0.2cm} \\
                  d_1+d_2+d_3 =  \tilde{d}_1+\tilde{d}_2+\tilde{d}_3.
                 
                \end{array}
              \right.
  \]

Clearly, with the condition that $d_1 \leq d_2 \leq d_3$ and $\tilde{d}_1 \leq \tilde{d}_2 \leq \tilde{d}_3$ the only solution for the above system is $d_i=\tilde{d}_i$ for $i=1,2$ and $3$, which completes the proof.\end{proof}\\

An interesting fact that we can find in the proof of the Theorem \ref{Main1} is that the weights and degrees of $f$ and $g$ are the same. We will record this result as a corollary below.

\begin{corollary}\label{weights}
    Let $f,g:(\mathbb{C}^2,0)\rightarrow(\mathbb{C}^3,0)$ be finitely determined quasihomogeneous map germs. If  $(f(\mathbb{C}^2),0)$ and $(g(\mathbb{C}^2),0)$ are topologically equivalent, then the weights and degrees of $f$ and $g$ are the same.
\end{corollary}

\begin{proof}
    The proof of this corollary is contained in the proof of the Theorem \ref{Main1}.
\end{proof}
	
	\section{The multiplicity of the image of a finitely determined, quasihomogeneous, corank 1 map germ}\label{sec4}
	
	$ \ \ \ \ $ In this section, we are interested in work with $f:(\mathbb{C}^n,0)\rightarrow(\mathbb{C}^{n+1},0)$ a finitely determined, quasihomogeneous, corank 1 map germs. We provide a positive answer to Question 2 in this setting. In what follows, we provide a quasihomogeneous normal form of a finitely determined, quasihomogeneous, corank 1 map germ from $(\mathbb{C}^n,0)$ to $(\mathbb{C}^{n+1},0).$

    \begin{lemma}\label{forma}
        Let $f:(\mathbb{C}^n,0)\rightarrow(\mathbb{C}^{n+1},0)$ be a finitely determined, quasihomogeneous, corank 1 map germ. Then, $f$ has a quasihomogeneous normal form given by \begin{equation}\label{Norm}
            f(x_1,\dots,x_n)=(x_1,\dots,x_{n-1},\alpha x_n^{m_1}+p(x_1,\dots,x_n),\beta x_n^{m_2}+q(x_1,\dots,x_n)),
        \end{equation} for some $\alpha,\beta\in\mathbb{C}$ not both zero, $p(x_1,\dots,x_{n-1},0)=0=q(x_1,\dots,x_{n-1},0)$ and $m_1\le m_2.$
    \end{lemma}

    \begin{proof}
        Write $f$ in a quasihomogeneous normal form $$f(x_1,\dots,x_n)=(x_1,\dots,x_{n-1},p_1(x_1,\dots,x_n),q_1(x_1,\dots,x_n)).$$ Without loss of generality, suppose that $f$ is of type $(w_1,\dots,w_{n-1},d_n,d_{n+1};w_1,\dots,w_n).$ Since $f$ is finitely determined, we have that $$p_1(x_1,\dots,x_n)=\alpha x_n^{m_1}+p(x_1,\dots,x_n)$$ and $$q_1(x_1,\dots,x_n)=\beta x_{n}^{m_2}+q(x_1,\dots,x_n),$$ for some $\alpha,\beta\in\mathbb{C}.$ We can assume that $m_1\le m_2.$ Furthermore, if $p(x_1,\dots,x_{n-1},0)\ne0$ or $q(x_1,\dots,x_{n-1},0)\ne0,$ we can apply isomorphims in the target to obtain $p(x_1,\dots,x_{n-1},0)=0=q(x_1,\dots,x_{n-1},0)$.
    \end{proof}

\begin{remark}
    In \rm{\cite[\textit{Lemma} 2.11]{normalform}}\textit{, the first author obtain the same normal form of $f$ in $n=2$ case, where in this case, $\alpha$ is always distinct to $0.$ However, when $n\ge3,$ there exists examples of finitely determined, quasihomogeneous, corank 1 map germs such that $\alpha$ may be zero in the normal form \eqref{Norm} in Lemma} \rm{\ref{forma}}\textit{. For example, in $n=3$ case, consider $f,g:(\mathbb{C}^3,0)\rightarrow(\mathbb{C}^4,0)$ given by $$f(x,y,z)=(x,y,xz,z^2+y^2z)$$ and $$g(x,y,z)=(x,y,xz^2+yz,z^3+x^6z).$$ One can verify that $f$ and $g$ are finitely determined, quasihomogeneous map germs, where $f$ is of type $(1,1,3,4;1,1,2)$ and $g$ is of type $(1,6,7,9;1,6,3).$ Note that weight of $z$ does not divide the degree of third coordinate of $f$ (respectively, $g$). These examples illustrate the fact that $\alpha$ may be zero in normal form \eqref{Norm} in Lemma} \rm{\ref{forma}} \textit{.}
\end{remark}

The next lemma provides a condition that ensures $\alpha\ne0$ in Lemma \ref{forma}.

\begin{lemma}
    Let $f:(\mathbb{C}^n,0)\rightarrow(\mathbb{C}^{n+1},0)$ be a finitely determined, quasihomogeneous, corank 1 map germ, with $n\ge 3$. Write $f$ in the quasihomogeneous normal form $$f(x_1,\dots,x_n)=(x_1,\dots,x_{n-1},\alpha x_{n}^{m_1}+p(x_1,\dots,x_n),\beta x_n^{m_2}+q(x_1,\dots,x_n)),$$ for some $\alpha,\beta\in\mathbb{C}$ not both zero, $p(x_1,\dots,x_{n-1},0)=0=q(x_1,\dots,x_{n-1},0)$ and $m_1\le m_2.$ If $m_2>n,$ then $\alpha\ne0.$
\end{lemma}

\begin{proof}
    Suppose that $\alpha=0.$ Therefore, $$f(0,\dots,0,x_n)=(0,\dots,0,\beta x_n^{m_2})$$ is generically $m_2$-to-1, which contradicts the finite determinacy of $f.$ Thus, $\alpha$ must be non zero.
\end{proof}\\

Before presenting the next result, let us recall the notion of topological triviality of a 1-parameter unfolding to map germs from $(\mathbb{C}^n,0)$ to $(\mathbb{C}^{n+1},0)$. The notion of topological triviality will be very useful for proving the main result of this section.

 \begin{definition}\label{defwe}
    Let $f:(\mathbb{C}^n,0)\rightarrow(\mathbb{C}^{n+1},0)$ be a finitely determined map germ, and let $\mathcal{F}:(\mathbb{C}^n\times\mathbb{C},0)\rightarrow(\mathbb{C}^{n+1}\times\mathbb{C},0)$ a 1-parameter unfolding of $f$ given by $\mathcal{F}(x,t)=(f_t(x),t)$ and $\mathcal{I}(x,t)=(f(x),t)$ the trivial unfolding of $f.$ We say that $\mathcal{F}$ is topologically equisingular (or topologically trivial) if there exists germs of homeomorphisms $$\mathcal{G}:(\mathbb{C}^n\times\mathbb{C},0)\rightarrow(\mathbb{C}^n\times\mathbb{C},0)\, \ \text{and}\, \ \mathcal{H}:(\mathbb{C}^{n+1}\times\mathbb{C},0)\rightarrow(\mathbb{C}^{n+1}\times\mathbb{C},0)$$ such that $\mathcal{G}$ and $\mathcal{H}$ are 1-parameter unfoldings of identity map germs from $n$-space and $(n+1)$-space, respectively, such that $\mathcal{I}=\mathcal{H}\circ \mathcal{F}\circ \mathcal{G}.$
\end{definition}

\begin{remark}
    Let $f:(\mathbb{C}^n,0)\rightarrow(\mathbb{C}^{n+1},0)$ be a finitely determined map germ. Consider $\mathcal{F}=(f_t,t)$ a 1-parameter unfolding of $f$. In general, verify that $\mathcal{F}$ is topologically trivial or not by definition is not a simple task. However, when $f$ is quasihomogeneous and $\mathcal{F}$ is of non-negative degree (i.e., for any term $\alpha$ in the deformation of $f_i$, the weighted degree of $\alpha$ is not smaller than the weighted degree of $f_i$), Damon proved that $\mathcal{F}$ is topologically trivial (see \rm{\cite[\textit{Cor.} 1]{Damon}}\textit{). This is a simple way to produce topologically trivial 1-parameter unfoldings in quasihomogeneous case.}
\end{remark}

The next lemma will be very important to prove the Theorem \ref{lemamult}, which provides an answer to Question 2 in corank 1 case.

\begin{lemma}\label{equi}
    Let $f:(\mathbb{C}^n,0)\rightarrow(\mathbb{C}^{n+1},0)$ be a finitely determined, quasihomogeneous, corank 1 map germ with $n\ge2.$ Let $\mathcal{F}=(f_t,t)$ a topologically trivial 1-parameter unfolding of $f$ such that $\mathcal{F}^{-1}(\left\{0\right\}\times T)=\left\{0\right\}\times T,$ where $T$ is the parameter space. If $\mathcal{F}$ is an unfolding of non-negative degree, then the family $\mathcal{F}(\mathbb{C}^n\times\mathbb{C})$ is equimultiple.
\end{lemma}

\begin{proof}
    The proof uses arguments analogous to those presented in \cite[Th. 5.3]{upper}.
\end{proof}  \\

Now, we will answer Question 2 in corank 1 case.

    \begin{theorem}\label{lemamult}
        Let $f=(f_1,\dots,f_n,f_{n+1}):(\mathbb{C}^n,0)\rightarrow(\mathbb{C}^{n+1},0)$ be a finitely determined, quasihomogeneous, corank 1 map germ, with $n\ge3.$ Consider that $f$ is written in the quasihomogeneous normal form $$f(x_1,\dots,x_n)=(x_1,\dots,x_{n-1},\alpha x_{n}^{m_1}+p(x_1,\dots,x_n),\beta x_n^{m_2}+q(x_1,\dots,x_n)),$$ for some $\alpha,\beta\in\mathbb{C}$ not both zero, $p(x_1,\dots,x_{n-1},0)=0=q(x_1,\dots,x_{n-1},0)$ and $m_1\le m_2.$ Suppose that $f$ is of type $(w_1,\dots,w_{n-1},d_n,d_{n+1};w_1,\dots,w_n).$ Thus\\

        \noindent(a) if $w_n$ divides $d_n,$ then $m(f(\mathbb{C}^n),0)=d_n/w_n.$\\

        \noindent(b) if $w_n$ does not divide $d_n,$ then $m(f(\mathbb{C}^n),0)=d_{n+1}/w_n.$\\

        \noindent In particular, the multiplicity of the image of $f$ is determined by the weights and degrees of $f.$
    \end{theorem}

    \begin{proof}
        (a) Suppose that $w_n$ divides $d_n.$ Now, assume that $\alpha\ne0.$ Let $\pi:(\mathbb{C}^{n+1},0)\rightarrow(\mathbb{C}^n,0),$ $$ \pi(X_1,\dots,X_{n+1})=(l_1(X_1,\dots,X_{n+1}),\dots,l_n(X_1,\dots,X_{n+1})),$$ be a generic projection in the sense of Remark \ref{remarkMulti}. For each $i=1,\dots,n,$ write $$l_i(X_1,\dots,X_{n+1})=\sum_{j=1}^{n+1}\alpha_{i,j}X_j,$$ where $\alpha_{i,j}$ are complex constants. By Remark \ref{remarkMulti} we have that $$m(f(\mathbb{C}^n),0)=\dim_{\mathbb{C}}\dfrac{\mathbb{C}\left\{x_1,\dots,x_n\right\}}{\left< \sum_{j=1}^{n+1}\alpha_{1,j}f_j(x_1,\dots,x_n),\dots,\sum_{j=1}^{n+1}\alpha_{n,j}f_j(x_1,\dots,x_n)\right>}.$$ Since $\alpha\ne0,$ we have that $V(f_1)\cap\dots\cap V(f_{n-1})\cap V(f_n)=0.$ Thus, applying an appropriate version of the generalized Bézout's theorem (see, for instance, \cite{Chirka}), we conclude that $$\dim_{\mathbb{C}}\dfrac{\mathbb{C}\left\{x_1,\dots,x_n\right\}}{\left< \sum_{j=1}^{n+1}\alpha_{1,j}f_j(x_1,\dots,x_n),\dots,\sum_{j=1}^{n+1}\alpha_{n,j}f_j(x_1,\dots,x_n)\right>}=\dfrac{d_n}{w_n}.$$ Now suppose that $\alpha=0.$ Consider the 1-parameter unfolding $\mathcal{F}=(f_t,t)$ of $f$ given by $$f_t(x_1,\dots,x_n)=(x_1,\dots,x_{n-1},tx_n^{d_n/w_n}+p(x_1,\dots,x_n),\beta x_n^{m_2}+q(x_1,\dots,x_n)).$$ Note that $\mathcal{F}$ is a deformation of non-negative degree. By \cite[Cor. 1]{Damon} $\mathcal{F}$ is topologically trivial. Furthermore, it's clear that $\mathcal{F}^{-1}(\left\{0\right\}\times T)=\left\{0\right\}\times T,$ where $T$ is the parameter space. Therefore, by Lemma \ref{equi} we conclude that $m(f_t(\mathbb{C}^n),0)$ is constant along $T.$ Since $f_t$ is quasihomogeneous for all $t\ne0$ sufficiently small, applying applying an appropriate version of the generalized Bézout's theorem \cite{Chirka}, we have $m(f_t(\mathbb{C}^n),0)=d_n/w_n,$ for all $t\ne0$ sufficiently small. Thus, $m(f(\mathbb{C}^n),0)=d_n/w_n.$\\

        \noindent(b) Suppose that $w_n$ does not divide $d_n.$ In this case, we have that $\alpha=0.$ Note that $$f(0,\dots,0,x_n)=(0,\dots,0,\beta x_n^{m_2})$$ and by finite determinacy of $f$ we obtain $m_2\le n.$ Note that we can write \begin{equation}\label{p1}
            f_n=p(x_1,\dots,x_n)=p_1(x_1,\dots,x_{n-1})x_n^{m_2-1}+\dots+p_{m_2-1}(x_1,\dots,x_{n-1})x_n,
        \end{equation} and $p_1$ must be non zero, otherwise, using divided differences, one can verify that $D^{k}(f)$ is not an ICIS of dimension $n-k+1$, for some $k\in\left\{2,\dots,m_2+1\right\},$ which is a contradiction, since $f$ is finitely determined (see \cite[Th. 2.14]{MararMond}). Now, by an isomorphism, we can write $$f=(x_1,\dots,x_{n-1},\beta x_n^{m_2}+q(x_1,\dots,x_n),f_n)$$ and use Mond-Pellikaan algorithm \cite{[12]} to obtain a presentation matrix of $f_*\mathcal{O}_n$ as an $\mathcal{O}_{n+1}$-module. Denote this matrix by $M$. Since the $\mathbb{C}$-vector space $\mathcal{O}_n/\langle x_1,\dots,x_{n-1},\beta x_n^{m_2}+q(x_1,\dots,x_n)\rangle$ has dimension $m_2,$ $M$ is a square matrix of order $m_2$. The determinant of $M$ provides an equation for the image of $f$. By Mond-Pellikaan algorithm \cite{[12]}, we can write $M=\lambda-X_{n+1} Id_{m_2},$ where $\lambda$ is a matrix of order $m_2$ described in the algorithm and $Id_{m_2}$ is the identity matrix of order $m_2$. Therefore, a monomial $X_{n+1}^{m_2}$ appears in the determinant of $M$ and we have $m(f(\mathbb{C}^n),0)\le m_2.$ On the other hand, recall from \eqref{p1} that $p_1\ne0.$ Furthermore, if $p_i\ne0,$ for some $i,$ then $p_i$ cannot be a constant polynomial, since otherwise we would obtain that $w_n$ divides $d_n.$ Thus, by the Mond-Pellikaan algorithm, one can verify that the entries of $\lambda$ are zero or polynomials non constant in indeterminates $X_1,\dots,X_{n+1}.$ Thus, every monomial of the determinant of $M$ has at least degree $m_2$. Hence $m(f(\mathbb{C}^n),0)=m_2=d_{n+1}/w_n.$
    \end{proof}\\

    In Theorem \ref{lemamult}, we provide a positive answer to Question 2 in corank 1 case. Based on Theorem \ref{lemamult}, we define the notion of \textit{good multiplicity} for a finitely determined, quasihomogeneous map germ from $(\mathbb{C}^n,0)$ to $(\mathbb{C}^{n+1},0)$.
    
    \begin{definition}\label{good}
        Let $f:(\mathbb{C}^n,0)\rightarrow(\mathbb{C}^{n+1},0)$ be a finitely determined, quasihomogeneous map germ of type $(d_1,\dots,d_{n+1};w_1,\dots,w_n).$ Consider the set $\left\{m_1,\dots,m_{n+1} \right\}$ where \begin{equation}\label{mi}
            m_i=\dfrac{d_1\dots d_{i-1}d_{i+1}\dots d_{n+1}}{w_1\dots w_n}.
        \end{equation} We say that the multiplicity of the image of $f$ is a good multiplicity if $m(f(\mathbb{C}^n),0)\in\left\{m_1,\dots,m_{n+1}\right\}\cap\mathbb{N}.$
    \end{definition}
    
    Note that if the image of $f$ has a good multiplicity, then that multiplicity can be expressed naturally in terms of the weights and degrees in the sense of \eqref{mi}. Thus, a natural formulation for the multiplicity of the image of a finitely determined, quasihomogeneous map germ in general case is given as in the question below:

    \begin{mybox}
		\textbf{Question 3:}  Let $f:(\mathbb{C}^n,0)\rightarrow(\mathbb{C}^{n+1},0)$ be a finitely determined, quasihomogeneous map germ, with $n\ge2.$ Denote by $(\textbf{X},0)=(f(\mathbb{C}^n),0).$ Is the multiplicity of $(\textbf{X},0)$ a good multiplicity?
	\end{mybox}

    If $n=2$, Lemma \ref{lema aux5}  provides a positive answer to Question 3. For $n\ge3,$ Theorem \ref{lemamult}  provides a positive answer to Question 3 in the corank 1 case. Apart from these cases, unfortunately the answer of the Question 3 may be negative. For instance, in Table \ref{exemplos}, we provide some examples of finitely determined, quasihomogeneous, corank 2 map germs from $(\mathbb{C}^3,0)$ to $(\mathbb{C}^{4},0)$ such that the multiplicity of the image is not a good multiplicity.

    \begin{table}[H]
\centering
{\def\arraystretch{2.1}\tabcolsep=12pt 

\begin{tabular}{ c | c | c }

\hline
\rowcolor{lightgray}
Map germ   & Multiplicity & Good multiplicity \linebreak possibilities \\ 

\hline
$\hat{A}_3(x,y,z)=(x,y^3+xz+x^4y,yz,z^2+y^5)$   & $5$  &  $6$ or $7$   \\
\hline
$\hat{A}_4(x,y,z)=(x,y^4+xz+x^6y,yz,z^2+y^5)$   & $6$  &  $8$ or $9$   \\
\hline
$f_{4,1}(x,y,z)=(x,yz,y^4+xz,z^2+x^2y^2+y^5)$   & $6$  &  $7$ or $8$   \\
\hline
$f_{5,2}(x,y,z)=(x,yz,y^5+xz,z^2+x^4y+y^7)$   & $7$ & $9$ or $10$ \\
\hline

$f_{7,1}(x,y,z)=(x,yz,y^7+xz,z^2+x^6y^2+y^{11})$ & $9$ & $13$ or $14$\\

\hline

$f_{11,2}(x,y,z)=(x,yz,y^{11}+xz,z^2+x^{12}y+y^{19})$ & $13$ & $21$ or $22$\\

\hline

$f_{13,1}(x,y,z)=(x,yz,y^{13}+xz,z^2+x^{14}y^2+y^{23})$ & $15$ & $25$ or $26$\\

\hline

$f_{16,1}(x,y,z)=(x,yz,y^{16}+xz,z^2+x^{18}y^2+y^{29})$ & $18$ & $31$ or $32$\\

\hline
\end{tabular}
}
\caption{Examples where the multiplicity of the image is not a good multiplicity.}\label{exemplos}
\end{table}

\begin{remark}
    (a) The map germs $\hat{A}_3$ and $\hat{A}_4$ in Table \rm{\ref{exemplos}} \textit{were exhibited by Altinta\c{s} Sharland in} \rm{\cite{Altintas}}\textit{. These map germs belong to the family of map germs $$\hat{A}_k(x,y,z)=(x,y^k+xz+x^{2k-2}y, yz, z^2+y^{2k-1}).$$ Note that $\hat{A}_k$ is of type $(1,2k,2k+1,2(2k-1);1,2,2k-1).$}\\

    \noindent\textit{(b) The map germs $f_{k,1}$ and $f_{k,2}$ in Table} \rm{\ref{exemplos}} \textit{belong to the following families of map germs: $$f_{k,1}(x,y,z)=(x,yz,y^k+xz,z^2+x^{4k'-2}y^2+y^{2k-3}),\ \ \text{if}\ \ k=3k'+1$$ and $$f_{k,2}(x,y,z)=(x,yz,y^k+xz,z^2+x^{4k'}y+y^{2k-3}),\ \ \text{if}\ \ k=3k'+2.$$ Note that both $f_{k,1}$ and $f_{k,2}$ are of type $(3,2k-1,2k,2(2k-3);3,2,2k-3).$ Furthermore, it is worth noting that the map germs $f_{k,1}$ and $f_{k,2}$ in the Table} \rm{\ref{exemplos}} \textit{support the Mond's conjecture} \rm{\cite{conjmond}}.\\

    \noindent\textit{(c) Consider the map germ $f:(\mathbb{C}^4,0)\rightarrow(\mathbb{C}^5,0)$ given by $$f(x,y,z,w)=(x,y,z^3+xw+y^4z,zw,w^2+x^8z+z^5).$$ One can verify that $f$ is finitely determined. Furthermore, $f$ is quasihomogeneous of type $(1,1,6,7,10;1,1,2,5)$ and $m(f(\mathbb{C}^4),0)=5$, which is not a good multiplicity for this map germ. Thus, $f$ provides a negative answer to the Question 3 for $n=4.$ It is worth noting that $f$ also supports Mond's conjecture} \rm{\cite{conjmond}}.
\end{remark}

    To conclude this section, we leave the following problem, based on the examples in Table \ref{exemplos}.

    \begin{mybox}
		\textbf{Problem 1:}  Let $n\ge3$ be an integer.\\
        
        \noindent(a) For each $r=2,\dots,n,$ is there a finitely determined, quasihomogeneous, corank $r$ map germ $f$ from $(\mathbb{C}^n,0)$ to $(\mathbb{C}^{n+1},0)$ such that the multiplicity of its image is not a good multiplicity?\\

        \noindent(b) If the image of $f$ does not have a good multiplicity, is it possible to determine $m(f(\mathbb{C}^n),0)$ in terms of the weights and degrees of $f$?
	\end{mybox}

	\section{The Zariski's multiplicity conjecture for $n$-varieties}\label{sec5}
	
	$ \ \ \ \ $  In this section, under some assumptions, we will use Theorem \ref{lemamult} to provide an answer to Zariski's multiplicity conjecture for $n$-varieties in $(n+1)$-space, for $3\le n \le 4.$ First, we will prove a lemma that guarantees that if a finitely determined map germ from $(\mathbb{C}^n,0)$ to $(\mathbb{C}^{n+1},0)$ is such that its double point space is smooth, then the map germ has corank 1, for all $n\ge2$.

    \begin{lemma}\label{D2}
        Let $f=(f_1,\dots,f_{n+1}):(\mathbb{C}^n,0)\rightarrow(\mathbb{C}^{n+1},0)$ be a finitely determined singular map germ, with $n\ge2$. If $D^2(f)$ is smooth, then $f$ has corank 1.
    \end{lemma}

    \begin{proof}
        Let us prove the case $n=2$. The argument for the general case follows analogously, requiring only the appropriate adjustments. Let $f=(f_1,f_2,f_3):(\mathbb{C}^2,0)\rightarrow(\mathbb{C}^3,0)$ be a finitely determined singular map germ. Let $\alpha=(\alpha_{ij})$ be a $3\times 2$ matrix with $\alpha_{ij}\in\mathcal{O}_{\mathbb{C}^2\times\mathbb{C}^2}$, described in Section \ref{2.1}, such that 
        \begin{equation}\label{eq}
            f_i(x,y)-f_i(x',y')=\alpha_{i1}(x,y,x',y')(x-x')+\alpha_{i2}(x,y,x',y')(y-y'), \ \ i=1,2,3.
        \end{equation}
         Thus $$D^2(f)=V(g_1,g_2,g_3,g_4,g_5,g_6)\subset\mathbb{C}\left\{x,y,x',y'\right\},$$ where $$g_i=f_i(x,y)-f_i(x',y'), \ \ i=1,2,3,$$ and for each $i=4,5,6,$ $g_i$ denotes a $2\times2$ minor of $\alpha.$ Now define $F:\mathbb{C}^4\rightarrow\mathbb{C}^6$ by $$F(x,y,x',y')=(g_1,g_2,g_3,g_4,g_5,g_6).$$ Consider the Jacobian matrix of $F$ given by $$Jac(F)=\left[\begin{array}{ccc}
        (g_1)_x & \cdots & (g_1)_{y'}\\
        \vdots & \ddots & \vdots \\
        (g_6)_x & \cdots & (g_6)_{y'}
        \end{array}\right].$$ 
        
        Denote by $Sing(D^2(f))$ the singular locus of $D^2(f).$ By singular locus notion (see, for instance, \cite[Sec. 5.7]{Greuel}), we have that $Sing(D^2(f))=V(J)\cap D^2(f),$ where $J$ is the ideal generated by $3\times 3$ minors of $Jac(F).$ But $D^2(f)$ is smooth. Hence, $Sing(D^2(f))=\emptyset.$ Since $$\emptyset=V(J+\langle g_1,\dots,g_6\rangle),$$ some generator of $J+\langle g_1,\dots,g_6\rangle$ is invertible in $\mathbb{C}\left\{x,y,x',y'\right\}.$ Let us show that $J$ contains an invertible generator. It's clear that $g_1,g_2$ and $g_3$ are not invertible. Suppose that $g_i$ is invertible, for some $i=4,5,6.$ Without loss of generality, suppose that $g_4$ is invertible, where $g_4=\alpha_{11}\alpha_{22}-\alpha_{12}\alpha_{21}.$ Since $g_4$ is invertible (meaning $g_4(0)\ne0$), at least one of the terms $\alpha_{11}\alpha_{22}$ or $\alpha_{12}\alpha_{21}$ must be non-zero at the origin (if both were zero, $g_4(0)$ would be zero, which is a contradiction). Again, without loss of generality, suppose that $\alpha_{11}\alpha_{22}$ is invertible. Therefore, $\alpha_{11}$ and $\alpha_{22}$ are invertible. Thus, we can write \begin{equation}\label{alfa1}
    \alpha_{11}=a+\widetilde{\alpha}_{11} \ \ \ \text{and} \ \ \ \alpha_{22}=b+\widetilde{\alpha}_{22},
\end{equation} with $a,b\ne0$ and $\widetilde{\alpha}_{11}(0)=0=\widetilde{\alpha}_{22}(0)$. Replacing \eqref{alfa1} in \eqref{eq}, we conclude that $f$ is not singular, which contradicts the assumptions. Therefore, at least one generator of $J$ is invertible. For example, if we consider that the $3\times 3$ minor of $Jac(F)$ given by $$\left|\begin{array}{ccc}
        (g_1)_x & (g_1)_y & (g_1)_{x'}\\
        (g_2)_x & (g_2)_y & (g_2)_{x'} \\
        (g_3)_x & (g_3)_y & (g_3)_{x'}
        \end{array}\right|$$ is invertible, and use the same ideas that we used in $g_4$, we conclude that $f$ is not singular. On other hand, if we consider the $3\times3$ minor of $Jac(F)$ given by $$\left|\begin{array}{ccc}
        (g_1)_x & (g_1)_y & (g_1)_{x'}\\
        (g_4)_x & (g_4)_y & (g_4)_{x'} \\
        (g_5)_x & (g_5)_y & (g_5)_{x'}
        \end{array}\right|$$ is invertible, one can verify that $f$ has corank 1. Now, analyzing the other generators of $J$, we obtain some generators that cannot be invertible and other that implies that $f$ has corank 1.
    \end{proof}\\

    The following lemma is essential to the proof of our main result in this section.

    \begin{lemma}\label{mu2}
        Let $f:(\mathbb{C}^n,0)\rightarrow(\mathbb{C}^{n+1},0)$ be a finitely determined, corank 1 map germ, with $n\ge2$. Write $f$ in normal form $$f(x_1,\dots,x_n)=(x_1,\dots,x_{n-1},p(x_1,\dots,x_n),q(x_1,\dots,x_n)).$$ If $D^2(f)$ is smooth, then $p$ and $q$ each have a monomial term of multiplicity 2. More precisely, $p$ and $q$ each have a monomial of the form $x_1x_n,\dots, x_{n-1}x_n$ or $x_n^2.$
    \end{lemma}

    \begin{proof}
        Again, let us prove the $n=2$ case. The general case is analogous, requiring only minor adjustments. Let $D^2(f)=V(\phi(x,y,z),\psi(x,y,z))$ be the double point set of $f$, where $\phi(x,y,z)$ and $\psi(x,y,z)$ are the divided differences of $p$ and $q,$ respectively, described in Section \ref{2.1}. Now, define $F:\mathbb{C}^3\rightarrow\mathbb{C}^2$ by $F(x,y,z)=(\phi,\psi).$ Consider the Jacobian matrix $Jac(F)$ of $F$ and let $J$ the ideal generated by $2\times2$ minors of Jacobian matrix of $F.$ Since $D^2(f)$ is smooth, we have that $$\emptyset=Sing(D^2(f))=V(J+\langle\phi,\psi\rangle),$$ which implies that some generator of $J+\langle\phi,\psi\rangle$ is invertible in $\mathbb{C}\left\{x,y,z\right\}.$ Thus, using analogous arguments of Lemma \ref{D2}, we can assume (without loss of generality) that $\phi_x$ and $\psi_y$ are invertible. Therefore, $\phi$ has a monomial $x$ and $\psi$ has a monomial $y.$ Hence, $p$ has a monomial $xy$ and $q$ has a monomial $y^2.$ 
    \end{proof}\\

    In the sequel, we will proof our second main result, Theorem \ref{T2}. Again, let us state Theorem \ref{T2} with the markings from this section.

    \begin{theorem}\label{TEO1}
        Let $f,g:(\mathbb{C}^n,0)\rightarrow(\mathbb{C}^{n+1},0)$ be finitely determined, quasihomogeneous, corank 1 map germs, with $3\le n \le 4$. Denote by $(\textbf{X},0)=(f(\mathbb{C}^n),0)$ and $(\textbf{Y},0)=(g(\mathbb{C}^n),0).$ Suppose that $(\textbf{X},0)$ and $(\textbf{Y},0)$ are topologically equivalent. If the multiple point spaces $D^{n-1}(f)$ and $D^{n-1}(g)$ are smooth, then $m(\textbf{X},0)=m(\textbf{Y},0).$
    \end{theorem}

    \begin{proof}
Let us first consider the case $n=3$. Since $f$ and $g$ have corank 1, we can write $f$ and $g$ in the normal forms $$f(x,y,z)=(x,y,p(x,y,z),q(x,y,z)) \ \ \ \text{and}\ \ \ g(u,v,w)=(u,v,r(u,v,w),s(u,v,w)),$$ where $p,q,r,s\in\mathfrak{m}_3^2$ and $\mathfrak{m}$ denotes the maximal ideal of $\mathcal{O}_3.$ Suppose that $f$ is of type $(w_1,w_2,d_3,d_4;w_1,w_2,w_3)$ with $d_3\le d_4$ and $g$ is of type $(\widetilde{w}_1,\widetilde{w}_2,\widetilde{d}_3,\widetilde{d}_4;\widetilde{w}_1,\widetilde{w}_2,\widetilde{w}_3)$ with $\widetilde{d}_3\le\widetilde{d}_4.$ Now consider a germ of homeomorphism $h:(\mathbb{C}^4,0)\rightarrow(\mathbb{C}^4,0)$ such that $h(\textbf{X},0)=(\textbf{Y},0).$ By \cite[Lemma 5]{lev}, there exists a germ of homeomorphism $\widetilde{h}:(\mathbb{C}^3,0)\rightarrow(\mathbb{C}^3,0)$ such that the following diagram \begin{figure}[H]
        \begin{center}
        \begin{tikzpicture}
            \draw (0,0) node {$(\mathbb{C}^3,0)$};
            \draw (-1.4,-2.5) node {$(\textbf{X},0)\subset$};
            \draw (-1.7,0) node {$(D(f),0)\subset$};
            \draw (4.9,-2.5) node {$\supset(\textbf{Y},0)$};
            \draw (5.2,0) node {$\supset(D(g),0)$};
            \draw[->] (1,0) to (2.5,0);
            \draw (3.5,0) node {$(\mathbb{C}^3,0)$};
            \draw[->] (3.5,-0.5) to (3.5,-2);
            \draw (3.5,-2.5) node {$(\mathbb{C}^4,0)$};
            \draw[->] (1,-2.5) to (2.5,-2.5);
            \draw (0,-2.5) node {$(\mathbb{C}^4,0)$};
            \draw[->] (0,-0.5) to (0,-2);
            \draw (1.75,0.5) node {$\widetilde{h}$};
            \draw (1.75,-2) node {$h$};
            \draw (-0.5,-1.25) node {$f$};
            \draw (4,-1.25) node {$g$};
        \end{tikzpicture}
        \end{center}
        \end{figure}

\noindent commutes and $\widetilde{h}(D(f),0)=(D(g),0).$ In particular, $(D(f),0)$ and $(D(g),0)$ are topologically equivalent. Now, consider the normalization maps $n_1:(D^2(f),0)\rightarrow(D(f),0)$ and $n_2:(D^2(g),0)\rightarrow(D(g),0)$ of $(D(f),0)$ and $(D(g),0),$ respectively. Since $D^2(f)$ and $D^2(g)$ are smooth, we have that $$(D^2(f),0)\simeq(\mathbb{C}^2,0)\simeq(D^2(g),0).$$ Furthermore, the only singularities that appear outside the
origin in a small neighborhood of zero $U$ of $D(f)$ (respectively, $V$ of $D(g)$) via $n_1$ (respectively, via $n_2$) are transversal double points. Thus, By Mather-Gaffney criterion \cite[Th. 2.1]{Wall}, we have that $n_1$ (respectively, $n_2$) is finitely determined. Moreover, $n_1$ and $n_2$ are quasihomogeneous, corank 1 map germs. By Theorem \ref{Main1}, we conclude that $$m(D(f),0)=m(D(g),0).$$ Furthermore, the weights and degrees of $n_1$ and $n_2$ are the same. 

We will now examine the possible forms of the components of the normalization map germs $n_1$ and $n_2$. By Lemma \ref{mu2}, we know that $p$ and $q$ in $f$ (respectively, $r$ and $s$ in $g$) each have a monomial term of multiplicity 2. Suppose that $p$ has a monomial term $z^2.$ Using Mond-Pellikaan algorithm \cite{[12]}, one can verify that $m(\textbf{X},0)=2.$ Suppose that neither $r$ nor $s$ has a monomial term $w^2$. Thus, by finite determinacy there exists a monomial term $w^{m_1}$ in $r$ or $s$, with $m_1\ge3.$ Using divided differences for $f,$ we have that the triple point space of $f$ consists only of the origin. On other hand, the triple point space of $g$ is a curve. Following the notation of \cite[Sec. 1]{[12]}, consider the complex analytic space $M_2(f)=V(Fitt_2(f_*\mathcal{O}_3))$, that is, the set
of the triple point in the target. By \cite[Lemma 3.9]{Kleiman}, we have that \begin{eqnarray*}
            (f\circ n_1)^*(M_2(f))&=&n_1^*\circ f^*(Fitt_2(f_*\mathcal{O}_3))\\
            &=&n_1^*(Fitt_1((n_1)_*\mathcal{O}_2))\\
            &=&D^2(n_1).
        \end{eqnarray*} Therefore, $(D^2(n_1),0)\simeq(D^3(f),0).$ By an analogous argument, we have that $(D^2(n_2),0)\simeq(D^3(g),0).$ Since the homeomorphism $\widetilde{h}$ preserves the singular set, we have that $\widetilde{h}^{-1}$ maps the double point set of $n_2$ (a curve) in the double point set $n_1$ (a point), which is a contradiction. Thus, either $r$ or $s$ has a monomial term $w^2.$ Again, using Mond-Pellikaan algorithm \cite{[12]}, we conclude that $m(\textbf{Y},0)=2.$ To conclude the proof in the case $n=3,$ we need to study the case where (without loss of generality) $p$ has a monomial term $xz,$ $q$ has a monomial term $yz,$ $r$ has a monomial term $uw$ and $s$ has a monomial term $vw.$ In this setting, one can verify that $n_1$ is of type $(d_3-w_3,d_4-w_3,w_3;w_3,w_3)$ and $n_2$ is of type $(\widetilde{d}_3-\widetilde{w}_3,\widetilde{d}_4-\widetilde{w}_3,\widetilde{w}_3;\widetilde{w}_3,\widetilde{w}_3).$ By Corollary \ref{weights}, we conclude that $d_3=\widetilde{d}_3, d_4=\widetilde{d}_4$ and $w_3=\widetilde{w}_3.$ Therefore, by Theorem \ref{lemamult} we conclude the proof.

Now consider the $n=4$ case. Since $f$ and $g$ have corank 1, and $D^3(f)$ and $D^3(g)$ are smooth, by \cite[Prop. 9.14]{[7]} we have that $D^2(f)$ and $D^2(g)$ are both smooth. Let $h_0:(\mathbb{C}^5,0)\rightarrow(\mathbb{C}^5,0)$ be a germ of homeomorphism such that $h_0(\textbf{X},0)=(\textbf{Y},0).$ By \cite[Lemma 5]{lev}, there exists $h_1:(\mathbb{C}^4,0)\rightarrow(\mathbb{C}^4,0)$ be a germ of homeomorphism such that the following diagram \begin{figure}[H]
        \begin{center}
        \begin{tikzpicture}
            \draw (0,0) node {$(\mathbb{C}^4,0)$};
            \draw (-1.4,-2.5) node {$(\textbf{X},0)\subset$};
            \draw (-1.7,0) node {$(D(f),0)\subset$};
            \draw (4.9,-2.5) node {$\supset(\textbf{Y},0)$};
            \draw (5.2,0) node {$\supset(D(g),0)$};
            \draw[->] (1,0) to (2.5,0);
            \draw (3.5,0) node {$(\mathbb{C}^4,0)$};
            \draw[->] (3.5,-0.5) to (3.5,-2);
            \draw (3.5,-2.5) node {$(\mathbb{C}^5,0)$};
            \draw[->] (1,-2.5) to (2.5,-2.5);
            \draw (0,-2.5) node {$(\mathbb{C}^5,0)$};
            \draw[->] (0,-0.5) to (0,-2);
            \draw (1.75,0.5) node {$h_1$};
            \draw (1.75,-2) node {$h_0$};
            \draw (-0.5,-1.25) node {$f$};
            \draw (4,-1.25) node {$g$};
        \end{tikzpicture}
        \end{center}
        \end{figure}

\noindent commutes and $h_1(D(f),0)=(D(g),0).$ In particular, $(D(f),0)$ and $(D(g),0)$ are topologically equivalent. Consider the normalization maps $f_1:(D^2(f),0)\rightarrow(D(f),0)$ and $g_1:(D^2(g),0)\rightarrow(D(g),0)$ of $(D(f),0)$ and $(D(g),0),$ respectively. As in $n=3$, $f_1$ and $g_1$ are finitely determined, quasihomogeneous, corank 1 map germs. Since $D^2(f)$ and $D^2(g)$ are both smooth, we have that $(D^2(f),0)\simeq(\mathbb{C}^3,0)\simeq(D^2(g),0).$ In what follows, consider $f_2:(D^2(f_1),0)\rightarrow(D(f_1),0)$ and $g_2:(D^2(g_1),0)\rightarrow(D(g_1),0)$ the normalization map germs of $(D(f_1),0)$ and $(D(g_1),0),$ respectively. By \cite[Lemma 3.9]{Kleiman}, we have that \begin{eqnarray*}
            (f\circ f_1\circ f_2)^*(M_2(f))&=&f_2^*\circ f_1^*\circ f^*(Fitt_2(f_*\mathcal{O}_4))\\
            &=&f_2^*\circ f_1^*(Fitt_1(f_1)_*\mathcal{O}_3)\\
            &=&f_2^*(Fitt_0((f_2)_*\mathcal{O}_{D^2(f_1)}))\\
            &=&D^2(f_1).
        \end{eqnarray*} Therefore, there is an isomorphism between $(D^2(f_1),0)$ and $(D^3(f),0).$ By hypothesis, $(D^3(f),0)$ is smooth, which implies that $(D^2(f_1),0)$ is also smooth. By an analogous argument, we have that $(D^2(g_1),0)$ is smooth. By $n=3$ case, we conclude that $m(D(f),0)=m(D(g),0)$ and the weights and degrees of $f_1$ and $g_1$ are the same. With analogous arguments, one can verify that the weights and degrees of $f$ and $g$ are the same. As before, we conclude that $m(\textbf{X},0)=m(\textbf{Y},0).$
    \end{proof}

    \begin{remark}
        (a) One might ask if Theorem \ref{TEO1} can be extended to $n>4$. However, a result by Giménez Conejero \rm{\cite[\textit{Th.} 4.8]{Conejero}}\textit{ shows that if $n>4,$ then $D^{n-1}(f)$ and $D^{n-1}(g)$ are singular or empty. This renders a natural generalization for $n>4$ meaningless, since no examples in these dimensions satisfy the required hypotheses.}\\

        \noindent\textit{(b) Unfortunately, the only examples of map germs satisfying the hypotheses of the Theorem} \rm{\ref{TEO1}} \textit{in case $n=4$ are the stable map germs. A way to see this is to use the techniques of Lemma} \rm{\ref{mu2}}\textit{. Since $D^3(f)$ is smooth, then $Sing(D^3(f))=\emptyset.$ Moreover, since $f$ has corank 1, then we can compute the generators of $D^3(f)$ by divided differences. Thus, writing $f$ in the normal form $$f(x,y,z,w)=(x,y,z,p(x,y,z,w),q(x,y,z,w)),$$ one can verify that either $p$ or $q$ has a monomial term $w^3,$ which implies that $f$ does not have quadruple points. Therefore, by} \rm{\cite[\textit{Th.} 2.14]{MararMond}}\textit{, we conclude that $f$ is stable.} 
        
        \textit{On the other hand, this does not occur for $n=3$. Consider the family of map germs $f_k:(\mathbb{C}^3,0)\rightarrow(\mathbb{C}^4,0)$ given by $$f_k(x,y,z)=(x,y,z^k+xz,z^{k+2}+yz),$$ with $k\ge3.$ Using divided differences and} \rm{\cite[\textit{Th.} 2.14]{MararMond}}\textit{, one can verify that $f_k$ is a finitely determined map germ for all $k\ge3$. Furthermore, $D^2(f_k)$ is smooth and $f_k$ is not a stable map germ for all $k\ge3.$ Using the techniques presented in} \rm{\cite{Altintas}} \textit{and SINGULAR} \rm{\cite{singular}}\textit{, program we checked that $f_k$ supports the Mond's conjecture} \rm{\cite{conjmond}} \textit{up to $k=40$.}
    \end{remark}

    Some natural questions arise at this point: First, does the result still hold if one of the spaces $D^{n-1}(f)$ or $D^{n-1}(g)$ is not smooth? Second, can the assumption of
corank 1 be changed to any corank? If $D^{n-1}(f)$ or $D^{n-1}(g)$ is singular, the answer is directly linked to Theorem \ref{Main1}. This connection, however, depends on the case of singular normalization of $(\textbf{X},0)$, which itself remains an open problem for us. 

For the corank $>1$ case, the problem is directly linked to the solution of Problem 1 in Section \ref{sec4}. A further complication is the inherent difficulty of working with the multiple point space when the corank is greater than 1. In our technique, it was essential to know the equations for the multiple point spaces explicitly in order to apply recursion arguments. We conclude this section with the following problem, based on the discussion above.

    \begin{mybox}
		\textbf{Problem 2:}  Let $f,g:(\mathbb{C}^n,0)\rightarrow(\mathbb{C}^{n+1},0)$ be finitely determined, quasihomogeneous map germs, with $n\ge 3$. Denote by $(f(\mathbb{C}^n),0)=(\textbf{X},0)$ and $(g(\mathbb{C}^n),0)=(\textbf{Y},0).$ If $(\textbf{X},0)$ and $(\textbf{Y},0)$ are topologically equivalent, is it true that $m(\textbf{X},0)=m(\textbf{Y},0)?$
	\end{mybox}

    \begin{remark}
        The authors used the software Surfer \rm{\cite{surfer}} \textit{to create the figures in the text}.
    \end{remark}

    \section*{Acknowledgments}	
$ \ \ \ \ $ This work constitutes a part of the second author's Ph.D. thesis at the Federal University of Paraíba under the supervision of Otoniel Nogueira da Silva to whom he would like to express his deepest gratitude. The first author acknowledges support by grant Universal 407454/2023-3 from Conselho Nacional de Desenvolvimento Científico e Tecnológico (CNPq). The second author acknowledges support by Coordenação de Aperfeiçoamento de Pessoal de Nível Superior (CAPES). Finally, the authors are immensely grateful to Roberto Giménez Conejero for the inspiring conversations and suggestions during the 13th Mini Workshop on Singularities, Geometry and Differential Equations.

	\begin{flushleft}
    $\bullet$ Silva, O. N.\\
		\textit{otoniel.silva@academico.ufpb.br}\\
		Universidade Federal da Paraíba, 58.051-900, João Pessoa, PB, Brazil.

        $ \ \ $\\

		$\bullet$ Silva Jr, M. M.\\
		\textit{mmsj@academico.ufpb.br}\\
		Universidade Federal da Paraíba, 58.051-900, João Pessoa, PB, Brazil.\\

	\end{flushleft}

\end{document}